\documentclass[10pt]{amsart}

\usepackage{graphicx}
\usepackage{amscd}
\usepackage{amsmath,amsfonts,amssymb,amsthm,mathrsfs,cite}
\usepackage[all]{xy}
\xyoption{curve}
\usepackage{fancybox}
\usepackage{color}
\usepackage{url}
\usepackage{geometry}
\newtheorem{thm}{Theorem}[section]

\newtheorem{cor}[thm]{Corollary}

\newtheorem{lem}[thm]{Lemma}
\newtheorem{exm}[thm]{Example}

\newtheorem{prop}[thm]{Proposition}
\newtheorem{rmk}[thm]{Remark}
\newtheorem{defn}[thm]{Definition}
\numberwithin{equation}{section}

\newcommand{\ik}{\operatorname {IKer}}

\newcommand{\D}{\operatorname{D}}
\newcommand{\modcat}[1]{#1\mbox{\rm -mod}}
\newcommand{\End}{\operatorname{End}}
\newcommand{\add}{\operatorname{add}}
\newcommand{\Ima}{\operatorname{Im}}
\newcommand{\Cok}{\operatorname{Coker}}
\newcommand{\Ker}{\operatorname{Ker}}
\newcommand{\Hom}{\operatorname{Hom}}
\newcommand{\Ext}{\operatorname{Ext}}
\newcommand{\injd}{\operatorname{inj.dim}}
\newcommand{\lra}{\longrightarrow }

\newcommand{\m}{\Lambda}
\title [Bimodule monomorphism categories and RSS equivalences]
{Bimodule monomorphism categories \\ and RSS equivalences via cotilting modules}
\author{Bao-Lin Xiong\ \ \ \ \ Pu Zhang$^*$\ \ \ \ \ Yue-Hui Zhang}
\thanks {$^*$The corresponding author.}

\begin{document}

\renewcommand{\thefootnote}{\alph{footnote}}
\setcounter{footnote}{-1} \footnote{Supported by the NSF of China
(11431010, 11301019, 11271251).}
\setcounter{footnote}{-1}


\begin{abstract} The
monomorphism category $\mathscr{S}(A, M, B)$ induced by a bimodule $_AM_B$ is the subcategory of $\Lambda$-mod consisting of $\left[\begin{smallmatrix}
X\\ Y\end{smallmatrix}\right]_{\phi}$ such that $\phi: M\otimes_B Y\rightarrow X$ is a monic $A$-map, where $\Lambda=\left[\begin{smallmatrix}
A&M\\0&B
\end{smallmatrix}\right]$. In general, it is not the monomorphism categories induced by quivers. It could describe the Gorenstein-projective $\m$-modules. This monomorphism category is a resolving subcategory of $\modcat{\Lambda}$ if and only if $M_B$ is projective. In this case, it has enough injective objects and Auslander-Reiten sequences, and can be also described as the left perpendicular category of a unique basic cotilting $\Lambda$-module. If $M$ satisfies the condition ${\rm (IP)}$, then  the stable category of $\mathscr{S}(A, M, B)$ admits a recollement of additive categories, which is in fact a recollement of singularity categories if $\mathscr{S}(A, M, B)$ is a {\rm Frobenius} category. Ringel-Schmidmeier-Simson equivalence between $\mathscr{S}(A, M, B)$ and its dual is introduced. If $M$ is an exchangeable  bimodule,
then an {\rm RSS} equivalence is given by a $\Lambda$-$\Lambda$ bimodule which is a two-sided cotilting $\Lambda$-module with a special property; and the Nakayama functor $\mathcal N_\m$ gives an {\rm RSS} equivalence if and only if both $A$ and $B$ are Frobenius algebras.

\vskip5pt

{\it Keywords$:$} \  monomorphism category induced by bimodule, Auslander-Reiten sequence, cotilting module, recollement of additive categories, exchangeable  bimodule,   {\rm RSS} equivalence, Frobenius algebra, Nakayama functor
\end{abstract}
\maketitle
\vspace{-20pt}
\section{\bf Introduction and preliminaries}

\subsection{} Throughout, algebras mean Artin algebras, modules are finitely generated, and a subcategory is a full subcategory closed under isomorphisms.
For an algebra $A,$ let $A$-mod (resp. ${\rm mod}A$) be the category of left (resp. right) $A$-modules. So there is a duality $\D:
A\mbox{-}{\rm mod} \rightarrow {\rm mod}A$.
\vskip5pt

This paper is to draw attention to the monomorphism category $\mathscr{S}(A, M, B)$
induced by an $A$-$B$-bimodule $M$. It is defined to be the subcategory of $\Lambda$-mod consisting of left $\Lambda$-modules $\left[\begin{smallmatrix}
X\\ Y\end{smallmatrix}\right]_{\phi}$ such that $\phi: M\otimes_B Y\rightarrow X$ is a monic $A$-map, where $\Lambda$ is the triangular matrix algebra $\left[\begin{smallmatrix} A&M\\0&B \end{smallmatrix}\right]$.
When $_AM_B = \ _AA_A$, it is the classical submodule category $\mathscr{S}(A)$ in [RS1-RS3]. This submodule category is initiated in [Bir]. C. Ringel and M. Schmidmeier [RS2] establish its Auslander-Reiten
theory; and D. Simson  ([S1]-[S3]) studies its representation type. By D. Kussin, H. Lenzing and H. Meltzer ([KLM1, KLM2]; see also [C]), it is related to the singularity theory. It has been generalized via quivers to the filtered chain category  and the separated monomorphism category ([S1-S3], [Z1], [LZ], [ZX]).  However, all these generalizations can not include monomorphism categories
induced by bimodules (this will be clarified in Example \ref{nonexchangeable}). Another motivation is that $\mathscr{S}(A, M, B)$ can describe the Gorenstein-projective $\Lambda$-modules ([Z2, Thms. 1.4., 2.2]).

\subsection{} To study $\mathscr{S}(A, M, B)$, first, we need it to be a resolving subcategory of $\Lambda$-mod.
So we work under the condition that $M_B$ is projective: this is a necessary and sufficient condition such that $\mathscr{S}(A, M, B)$ is a resolving subcategory.
Then $\mathscr{S}(A, M, B)$ has enough projective objects and enough injective objects, and it is a Frobenius category if and only if $A$ and $B$ are selfinjective and $_AM$ and $M_B$ are projective (Corollary \ref{corollary-frobenius-cat}). This monomorphism category $\mathscr{S}(A, M, B)$ enjoys the functorially finiteness and {\rm Auslander-Reiten} sequences, and it closely relates to the tilting theory.
Here we use the classical cotilting modules of injective dimension at most $1$ ([HR], [R, p.167], [AR], [ASS, p.242]).
For a left $\Lambda$-module $Z$, let $^\perp Z$ denote the subcategory $\{L\in \Lambda\mbox{-}{\rm mod} \ | \ {\rm Ext}^m_\Lambda(L, Z) = 0, \ \forall \ m\ge 1\}.$

\begin{thm} \label{cotilting}
Let $M$ be an $A$-$B$-bimodule. Then

\vskip5pt

$(1)$ \ The following are equivalent$:$

\hskip20pt ${\rm (i)}$ \  $M_B$ is projective$;$

\hskip20pt ${\rm (ii)}$ \  $\mathscr{S}(A, M, B)$ is a resolving subcategory of $\modcat{\Lambda};$

\hskip20pt ${\rm (iii)}$ \ $_\Lambda T:=\left[\begin{smallmatrix}\D(A_A)\\ 0 \end{smallmatrix}\right]\oplus\left[\begin{smallmatrix}E_{\D(B_B)}\\ \D(B_B) \end{smallmatrix}\right]_{e}$ is a unique cotilting left $\Lambda$-module, up to multiplicities of indecomposable direct summands, such that $\mathscr{S}(A, M, B)={}^\perp T$, where  $E_{\D(B_B)}$ is an injective envelope of the left $A$-module $_AM\otimes_B{\D(B_B)}$ with inclusion $e: M\otimes_B{\D(B_B)}\hookrightarrow E_{\D(B_B)}$.

\vskip5pt

$(2)$ \ $\mathscr{S}(A, M, B)$ is a contravariantly finite subcategory of $\Lambda$-{\rm mod}. Moreover, if $M_B$ is projective, then $\mathscr{S}(A, M, B)$ is a functorially finite subcategory of $\Lambda$-{\rm mod}, and has {\rm Auslander-Reiten} sequences.
\end{thm}

\begin{cor} \label{cotilting2} \ If $_AM_B$ satisfies the condition {\rm (IP)}, then $_\Lambda T=\left[\begin{smallmatrix}\D(A_A)\\ 0 \end{smallmatrix}\right]\oplus\left[\begin{smallmatrix}M\otimes_B \D(B)\\ \D(B_B) \end{smallmatrix}\right]_{\rm Id}$ is a unique cotilting left $\Lambda$-module, up to multiplicities of indecomposable direct summands, such that $\mathscr{S}(A, M, B)={}^\perp T$.
\end{cor}

\subsection{} A recollement is first introduced for triangulated categories ([BBD]), and then for
abelian categories ([MV], [PS], [Ku]). It becomes a powerful tool in triangulated categories and in abelian categories (see e.g. [K\"o], [H2], [IKM], [FP], [PV], [FZ]).
One can also consider recollements of additive categories in the similar way.
For a subcategory $\mathscr{X}$ of an additive category $\mathscr{A}$, recall that the objects of
the stable category $\mathscr{A}/\mathscr{X}$ are the objects of $\mathscr{A}$, and
$\Hom_{\mathscr{A}/\mathscr{X}}(M, N): = \Hom_{\mathscr{A}}(M, N)/(M, \mathscr{X}, N)$, where
$(M, \mathscr{X}, N)$ is the subgroup consisting of those morphisms factoring through objects of
$\mathcal{X}$. For an algebra $A$, denote $A\mbox{-}{\rm mod}/{\rm inj}(A)$ by  $A\mbox{-}\overline{\rm mod}$, where ${\rm inj}(A)$ is the subcategory
of the injective $A$-modules. Similarly, $\overline{\mathscr{S}(A, M, B)}$ is the stable category of $\mathscr{S}(A, M, B)$ respect to the subcategory
of the injective objects of $\mathscr{S}(A, M, B)$.

\begin{thm} \label{theorem-recollement-additive-cat} An $A$-$B$-bimodule satisfying the condition {\rm (IP)} induces a recollement of additive categories

\begin{center}
\begin{picture}(100,48)
\put(-10,20){\makebox(-22,2) {$A\mbox{-}\overline{\rm mod}$}}
\put(30,30){\vector(-1,0){30}}
\put(0,20){\vector(1,0){30}}
\put(30,10){\vector(-1,0){30}}
\put(50,20){\makebox(25,0.8) {$\overline{\mathscr{S}(A, M, B)}$}}
\put(95,20){\vector(1,0){30}}
\put(125,10){\vector(-1,0){30}}
\put(125,30){\vector(-1,0){30}}
\put(135,20){\makebox(25,0.5){$B\mbox{-}\overline{\rm mod}$.}}
\put(235,20){\makebox(25,0.5){$(1.1)$}}
\put(15,35){\makebox(3,1){\scriptsize$\overline{i^*}$}}
\put(15,24){\makebox(3,1){\scriptsize$\overline{i_*}$}}
\put(15,14){\makebox(3,1){\scriptsize$\overline{i^!}$}}
\put(109,35){\makebox(3,1){\scriptsize$\overline{j_!}$}}
\put(109,24){\makebox(3,1){\scriptsize$\overline{j^*}$}}
\put(109,14){\makebox(3,1){\scriptsize$\overline{j_*}$}}
\end{picture}
\end{center}
If in addition $A$ and $B$ are selfinjective algebras, then it is in fact a recollement of singularity categories.
\end{thm}

Here the singularity category $\mathcal D^b_{sg}(\Lambda)$ of an algebra $\Lambda$ is defined to be the Verdier quotient $\mathcal D^b_{sg}(\Lambda): = \mathcal D^b(\Lambda\mbox{-}{\rm mod})/K^b({\rm proj}(\Lambda))$,
where $\mathcal D^b(\Lambda\mbox{-}{\rm mod})$ is the bounded derived category, and $K^b({\rm proj}(\Lambda))$ is the bounded homotopy category. See R. Buchweitz [Buch] and  D. Orlov [O].

\subsection{} The dual of $\mathscr{S}(A, M, B)$ is the epimorphism category $\mathscr F (A, M, B)$. The right module version of $\mathscr S(A, M, B)$ is $\mathscr S(A, M, B)_r$, and $\mathscr S(A, M, B)_r$ is a resolving subcategory if and only if $_AM$ is projective; in this case, there is a unique basic cotilting right $\Lambda$-module $U$ such that $\mathscr{S}(A, M, B)_r ={}^\perp (U_\Lambda)$.
Then $\mathscr F (A, M, B)$ can be also described as $\D\mathscr{S}(A, M, B)_r$.
Ringel-Schmidmeier-Simson equivalence $\mathscr{S}(A, M, B)\cong \mathscr{F}(A, M, B)$ is studied. Such an equivalence implies a strong symmetry, and was first observed by C. Ringel and M. Schidmeier [RS2] for the case of $_AM_B = \ _AA_A$, and by D. Simson [S1] for a chain without relations, and then developed to acyclic quivers with monomial relations in [ZX].

\vskip5pt

We introduce exchangeable bimodules. If $_AM_B$ is exchangeable, then the unique left cotilting $\Lambda$-module $T$ with
$\mathscr{S}(A, M, B) = \ ^\perp T$ (cf. Corollary \ref{cotilting2}) can be endowed with a $\Lambda$-$\Lambda$-bimodule structure via the exchangeable bimodule isomorphism, such that the right module $T_\Lambda$ coincides with the unique right cotilting $\Lambda$-module $U$ with $\mathscr{S}(A, M, B)_r ={}^\perp (U_\Lambda)$.  This two-sided cotilting $\Lambda$-module $_\Lambda T_\Lambda$ enjoys
a good property in the sense that $\End_\Lambda(_\Lambda T)^{op}\cong \Lambda$ as algebras, and under this isomorphism of algebras,
$T_{\End_\Lambda(_\Lambda T)^{op}}$ coincides with $T_\Lambda$. These good properties of $T$ induce an {\rm RSS} equivalence:

\begin{thm} \label{RSS1} \ Let $_AM_B$ be an exchangeable  bimodule. Then $T =\left[\begin{smallmatrix}\D(A)\\ 0 \end{smallmatrix}\right]\oplus\left[\begin{smallmatrix}M\otimes_B \D(B)\\ \D(B) \end{smallmatrix}\right]_{\rm Id}$ can be endowed with a $\Lambda$-$\Lambda$-bimodule such that  $\D\Hom_\m(-, T):
\mathscr S(A, M, B)\cong \mathscr F(A, M, B)$ is an {\rm RSS} equivalence. \end{thm}

The Nakayama functor $\mathcal N_\m$ gives an {\rm RSS} equivalence if and only if both $A$ and $B$ are Frobenius algebras (Proposition \ref{NakayamaandRSS}). Examples show that if $_AM_B$ is not exchangeable, then an {\rm RSS} equivalence can not be guaranteed.
Examples also show that the monomorphism category $\mathscr S(A, M, B)$ is not the separated monomorphism category of
the corresponding quiver in the sense of [ZX], in general (see Example \ref{nonexchangeable}). However, we do not know a sufficient and necessary condition and the uniqueness of an RSS equivalence. See Subsection 5.5.

\subsection{} Let $M$ be an $A$-$B$-bimodule. The multiplication of the associated matrix algebra
$\Lambda =\left[\begin{smallmatrix}
A&M\\0&B
\end{smallmatrix}\right]$ is given by
$\left[\begin{smallmatrix}
a & m\\ 0 & b
\end{smallmatrix}\right]\left[\begin{smallmatrix}
a' & m'\\ 0 & b'
\end{smallmatrix}\right]=\left[\begin{smallmatrix}
aa' & am'+mb'\\ 0 & bb'
\end{smallmatrix}\right]$.  Each left $\Lambda$-module is identified with a triple $\left[\begin{smallmatrix}
X\\ Y
\end{smallmatrix}\right]_{\phi}$, where $X\in A$-mod, $Y\in B$-mod, and $\phi: M\otimes_BY\rightarrow X$ is an $A$-map; and a $\Lambda$-map is identified with a pair $\left[\begin{smallmatrix}
f_1 \\ f_2
\end{smallmatrix}\right]:\left[\begin{smallmatrix}
X_1\\ Y_1
\end{smallmatrix}\right]_{\phi_1}\rightarrow \left[\begin{smallmatrix}
X_2\\ Y_2
\end{smallmatrix}\right]_{\phi_2}$, where $f_1: X_1\rightarrow X_2$ is an $A$-map, and $f_2: Y_1\rightarrow Y_2$ a $B$-map, such that the diagram
$$\xymatrix@R=0.5cm{
M\otimes_BY_1\ar[d]_{\phi_1}\ar[r]^{1\otimes f_2} &M\otimes_BY_2\ar[d]^{\phi_2}\\
X_1\ar[r]^{f_1} & X_2}$$
commutes. Under this identification, the indecomposable projective $\Lambda$-modules are
exactly $\left[\begin{smallmatrix}
   P \\
   0
\end{smallmatrix}\right]$ and $\left[\begin{smallmatrix}
   M\otimes_B Q  \\
   Q
\end{smallmatrix}\right]_{\rm Id}$, where $P$ and $Q$ run over the indecomposable projective $A$-modules and $B$-modules, respectively. The indecomposable injective $\Lambda$-modules are $\left[\begin{smallmatrix}
   I  \\
   {\rm Hom}_A (M,I)
\end{smallmatrix}\right]_{\varphi}
$ and $ \left[\begin{smallmatrix}
   0  \\
   J
\end{smallmatrix}\right]$, where $I$ and $J$ run over the indecomposable injective $A$-modules and $B$-modules, respectively ([ARS, p.73]).
Throughout, for any left $A$-module $X$, we denote by $\varphi = \varphi_X$ the left $A$-map $M\otimes_B \Hom_A(M, X) \longrightarrow X$ given by
$\varphi(m\otimes f) = f(m),$ i.e., the adjunction isomorphism $\Hom_A(M\otimes_B \Hom_A(M, X), X)\cong \Hom_B(\Hom_A(M, X), \Hom_A(M, X))$ sends $\varphi$ to ${\rm Id}_{{\rm Hom}_A(M, X)}$.
We will call $\varphi$ the involution map.

\subsection{\bf Conditions on a bimodule} A bimodule $_AM_B$  satisfies the condition {\rm (IP)}, if $M\otimes_B\D(B_B)$ is an injective left $A$-module and $M_B$ is projective.

\vskip5pt

A bimodule $_AM_B$ is {\it exchangeable}, if both $_AM$ and $M_B$ are projective and there is an $A$-$B$-bimodule isomorphism
$\D(_A A_A)\otimes_AM\cong M\otimes_B\D(_B B_B),$ which is called {\it an exchangeable  bimodule isomorphism}.

\vskip5pt

By adjunction isomorphisms we have $A$-$B$-bimodule isomorphisms $$\D(\D A\otimes_AM)\cong \Hom_A(M, A), \ \ \D(M\otimes_B\D B)\cong \Hom_B(M, B).$$
Thus $_AM_B$ is exchangeable  if and only if there is an $A$-$B$-bimodule isomorphism $\Hom_A(M, A)\cong\Hom_B(M, B)$; and if and only if
there is an $A$-$B$-bimodule isomorphism $\mathcal N_A(_AM) \cong \mathcal N_B(M_B)$, where $\mathcal N_A$ denotes the Nakayama functor $\D\Hom_A(-, A)$.

\vskip5pt

For $_AN$, let ${\rm add}(N)$ be the
subcategory of $A$-mod of direct summands of finite direct sums of
$N$.

\begin{exm} \label{mainexample} $(1)$ \ An exchangeable  bimodule $_AM_B$ satisfies the condition ${\rm (IP)}$, and $\D(_AA)\otimes_AM$ is an injective right $B$-module.

\vskip5pt

In fact, since $_AM$ is projective, $\D(_A A_A)\otimes_AM\in {\rm add}(\D(_A A_A)\otimes_AA) = {\rm add}(\D(A_A))$, so $\D(_A A_A)\otimes_AM$ is an injective left $A$-module. Thus $M\otimes_B\D(_B B_B)\cong \D(_A A_A)\otimes_AM$ is an injective left $A$-module. Similarly, one can prove that $\D(_AA)\otimes_AM$ is an injective right $B$-module.

\vskip5pt

$(2)$ \ If $B=A$ and $M=A\oplus\cdots\oplus A$, then $_AM_A$ is an exchangeable  bimodule.

\vskip5pt

{\rm (3)} \ For an algebra $A$, let $B: = A\oplus \cdots \oplus A,$ and $_AM_B: = \ _AB_B$. Then $_AM_B$ is an exchangeable  bimodule.

\vskip5pt

{\rm (4)} \ For an algebra $B$, let $A: = B\oplus \cdots \oplus B, $ and $_AM_B: = \ _AA_B$. Then $_AM_B$ is an exchangeable  bimodule.

\vskip5pt

{\rm (5)} \ An algebra $A$ over field $k$ is symmetric, if $\D(_AA_A) \cong \ _AA_A$ as $A$-$A$-bimodules. If both $A$ and $B$ are symmetric algebras and $M = \ _AP\otimes_kQ_B$, where $_AP$ and $Q_B$ are projective, then $_AM_B$ is an exchangeable  bimodule.

\vskip5pt

$(6)$ \ An algebra $B$ is {\rm Frobenius}, if $\D(B_B)\cong{_BB}$ as left $B$-modules.
If $B$ is a {\rm Frobenius} algebra and $_AM_B$ is a bimodule with $_AM$ injective and $M_B$ projective, then $_AM_B$ satisfies the condition ${\rm (IP)}$.

\vskip5pt

$(7)$ \ Let $B$ be a selfinjective algebra, and $_AM_B$ a bimodule with $M_B$ projective. Then
$_AM_B$ satisfies the condition ${\rm (IP)}$ if and only if $_AM$ is injective.
In particular, if both $A$ and $B$ are selfinjective $k$-algebras and $M = \ _AP\otimes_kQ_B$, where $_AP$ and $Q_B$ are projective, then $_AM_B$ satisfies the condition ${\rm (IP)}$.

\vskip5pt

In fact, since $B$ is a selfinjective algebra,  $\D(B_B) \in {\rm add} (_BB)$, and
hence $_AM\otimes_B\D(B_B)\in {\rm add} (_AM)$. So $M\otimes_B\D(B_B)$  is an injective $A$-module if and only if $_AM$ is injective.
\end{exm}

\section{\bf Monomorphism categories induced by bimodules}

\subsection{} Recall that {\it the monomorphism category} $\mathscr{S}(A, M, B)$ induced by bimodule $_AM_B$
is the subcategory of $\Lambda$-mod consisting of
$\left[\begin{smallmatrix}X\\ Y\end{smallmatrix}\right]_{\phi}$
such that $\phi: M\otimes_B Y\longrightarrow X$ is a monic $A$-map.
So it contains all the projective $\Lambda$-modules and is closed under direct sums and direct summands.

\begin{lem} \label{scat0} \ Let $_AM_B$ be a bimodule. Then $\mathscr{S}(A, M, B)$ is closed under extensions.
Thus $\mathscr{S}(A, M, B)$ is an exact category with the canonical exact structure$,$ and hence {\rm a Krull-Schmidt} category.
\end{lem}
\begin{proof} \ For an exact sequence
$0\rightarrow {\left[\begin{smallmatrix}
X_1\\ Y_1
\end{smallmatrix}\right]_{\phi_1}}\stackrel {\left[\begin{smallmatrix}
f_1 \\ g_1
\end{smallmatrix}\right]}\longrightarrow  {\left[\begin{smallmatrix}
X\\ Y
\end{smallmatrix}\right]_{\phi}}\stackrel {\left[\begin{smallmatrix}
f_2 \\ g_2
\end{smallmatrix}\right]}\longrightarrow {\left[\begin{smallmatrix}
X_2\\ Y_2
\end{smallmatrix}\right]_{\phi_2}}\rightarrow 0$ in $\Lambda$-mod with $\left[\begin{smallmatrix}
X_1\\ Y_1
\end{smallmatrix}\right]_{\phi_1}\in \mathscr{S}(A, M, B)$ and $\left[\begin{smallmatrix}
X_2\\ Y_2
\end{smallmatrix}\right]_{\phi_2}\in \mathscr{S}(A, M, B)$, we get a commutative diagram
$$\xymatrix @R=0.4cm {
& M\otimes_B Y_1\ar[d]^{\phi_1}\ar[r]^{1\otimes g_1} &M\otimes_BY\ar[d]^{\phi}\ar[r]^{1\otimes g_2} &M\otimes_BY_2\ar[d]^{\phi_2}\ar[r] & 0\\
0\ar[r] & X_1\ar[r]^-{f_1} & X\ar[r]^{f_2} & Y_2\ar[r] & 0}$$
with exact rows. It follows from the Snake Lemma that $\phi$ is monic, i.e., $\left[\begin{smallmatrix}
X\\ Y
\end{smallmatrix}\right]_{\phi}\in \mathscr{S}(A, M, B).$ \end{proof}

\begin{prop} \label{enoughprojinjobj} \ Let $_AM_B$ be a bimodule such that $M_B$ is projective. Then

\vskip5pt

$(1)$ \ $\mathscr{S}(A, M, B)$ has enough projective objects$;$ and the indecomposable projective objects of $\mathscr{S}(A, M, B)$ are exactly the indecomposable projective $\Lambda$-modules.

\vskip5pt

$(2)$ \ $\mathscr{S}(A, M, B)$ has enough injective objects$;$ and the indecomposable injective objects are exactly $\left[\begin{smallmatrix}I\\ 0 \end{smallmatrix}\right]$ and $\left[\begin{smallmatrix}E_J\\ J \end{smallmatrix}\right]_{e}$, where $I$ $($resp. $J$$)$ runs over indecomposable injective left $A$-modules $($resp. $B$-modules$)$, and $E_J$ is an injective envelope of the left $A$-module $M\otimes_BJ$ with inclusion $e: M\otimes_BJ\hookrightarrow E_J$.

\vskip5pt

In particular, if $M$ satisfies the condition {\rm (IP)}, then the indecomposable injective objects of $\mathscr{S}(A, M, B)$ are exactly $\left[\begin{smallmatrix} I \\ 0 \end{smallmatrix}\right]$ and $\left[\begin{smallmatrix}M\otimes_BJ\\ J \end{smallmatrix}\right]_{\rm Id}$.
\end{prop}
\begin{proof} $(1)$ \ Projective $\Lambda$-modules are clearly projective objects of $\mathscr{S}(A, M, B)$.
 For any object $\left[\begin{smallmatrix}X\\ Y \end{smallmatrix}\right]_{\phi}\in\mathscr{S}(A, M, B)$, taknig projective covers $\pi_Y: \ _BQ\to _BY$ and $\pi_C: \ _AP\to \ _A\Cok(\phi)$, we get exact sequences $0\rightarrow \Ker(\pi_Y)\stackrel {i_Y}\rightarrow Q\stackrel {\pi_Y} \rightarrow Y\rightarrow 0$ and $0\rightarrow \Ker(\pi_C)\stackrel {i_C} \rightarrow P\stackrel{\pi_C} \rightarrow \Cok(\phi)\rightarrow 0$.
 Consider the exact sequence $0 \longrightarrow M\otimes_B Y \stackrel \phi\longrightarrow X \stackrel \pi \longrightarrow {\rm Coker} (\phi) \longrightarrow 0.$
We get an $A$-map $\theta: P\to X$ such that $\pi_C=\pi\theta$, and hence the commutative diagram in $A$-mod with exact rows
$$\xymatrix@R=0.5cm{0\ar[r] & M\otimes_BQ\ar[r]^-{\left[\begin{smallmatrix} 1\\ 0 \end{smallmatrix}\right]}\ar[d]_-{1\otimes\pi_Y} & (M\otimes_BQ)\oplus P\ar[r]^-{[0, 1]}\ar[d]_-{[\phi(1\otimes\pi_Y), \theta]} & P\ar[r]\ar[d]^{\pi_C} & 0 \\
0\ar[r] & M\otimes_BY\ar[r]_-{\phi} & X\ar[r]_-{\pi} & \Cok(\phi)\ar[r] & 0.}$$ Since $M_B$ is projective,
$0\longrightarrow M\otimes_B\Ker(\pi_Y)\stackrel{1\otimes{i_Y}}\longrightarrow M\otimes_BQ \stackrel{1\otimes{\pi_Y}}\longrightarrow M\otimes_BY\longrightarrow 0$ is an exact sequence of $A$-module. By the Snake Lemma
we get the commutative diagram in $A$-mod with exact rows and columns
$$\xymatrix@R=0.5cm{& 0\ar[d] & 0\ar[d] & 0\ar[d] & \\
0\ar[r] & M\otimes_B\Ker(\pi_Y)\ar@{-->}[r]^-{\psi}\ar[d]_-{1\otimes{i_Y}} & N\ar@{-->}[r]\ar@{-->}[d] & \Ker(\pi_C)\ar[r]\ar[d]^-{i_C} & 0 \\
0\ar[r] & M\otimes_BQ\ar[r]^-{\left[\begin{smallmatrix} 1\\ 0 \end{smallmatrix}\right]}\ar[d]_-{1\otimes\pi_Y} & (M\otimes_BQ)\oplus P\ar[r]^-{[0, 1]}\ar[d]_-{[\phi(1\otimes\pi_Y), \theta]} & P\ar[r]\ar[d]^{\pi_C} & 0 \\
0\ar[r] & M\otimes_BY\ar[r]_-{\phi}\ar[d] & X\ar[r]_-{\pi}\ar[d] & \Cok(\phi)\ar[r]\ar[d] & 0 \\
& 0 & 0 & 0 & }$$
So the left and middle columns give the exact sequence
$$\xymatrix{0\ar[r] & {\left[\begin{smallmatrix} N \\ \Ker(\pi_Y) \end{smallmatrix}\right]_{\psi}}\ar[r] & {\left[\begin{smallmatrix}P\\ 0 \end{smallmatrix}\right]\oplus\left[\begin{smallmatrix}M\otimes_BQ\\ Q \end{smallmatrix}\right]_{\rm Id}}\ar[r] & {\left[\begin{smallmatrix}X\\ Y \end{smallmatrix}\right]_{\phi}}\ar[r] & 0}  \eqno(*)$$
in $\mathscr{S}(A, M, B)$. This shows that $\mathscr{S}(A, M, B)$ has enough projective objects.
\vskip5pt
Let $\left[\begin{smallmatrix}X\\ Y \end{smallmatrix}\right]_{\phi}$ be an indecomposable projective object of $\mathscr{S}(A, M, B)$. Then the exact sequence $(*)$ splits and $\left[\begin{smallmatrix}X\\ Y \end{smallmatrix}\right]_{\phi}$ is a direct summand of $\left[\begin{smallmatrix}P\\ 0 \end{smallmatrix}\right]\oplus\left[\begin{smallmatrix}M\otimes_BQ\\ Q \end{smallmatrix}\right]_{\rm Id}$. By Lemma \ref{scat0}, $\mathscr{S}(A, M, B)$ is a Krull-Schimdt category, so $\left[\begin{smallmatrix}X\\ Y \end{smallmatrix}\right]_{\phi}$ is isomorphic to $\left[\begin{smallmatrix}P'\\ 0 \end{smallmatrix}\right]$ or $\left[\begin{smallmatrix}M\otimes_BQ'\\ Q' \end{smallmatrix}\right]_{\rm Id},$ where $P'$ (resp. $Q'$) is an indecomposable projective $A$-module (resp. $B$-module). Thus $\left[\begin{smallmatrix}X\\ Y \end{smallmatrix}\right]_{\phi}$ is a projective $\Lambda$-module.

\vskip5pt

$(2)$ \ It is clear that $\left[\begin{smallmatrix}I\\ 0 \end{smallmatrix}\right]$ and $\left[\begin{smallmatrix}E_J\\ J \end{smallmatrix}\right]_{e}$ are indecomposable objects of $\mathscr{S}(A, M, B)$.
We show that they are injective objects of $\mathscr{S}(A, M, B)$. Put  $\left[\begin{smallmatrix}W\\ V \end{smallmatrix}\right]_{\phi}$ to be $\left[\begin{smallmatrix}I\\ 0 \end{smallmatrix}\right]$ or $\left[\begin{smallmatrix}E_J\\ J \end{smallmatrix}\right]_{e}$. For an exact sequence in $\mathscr{S}(A, M, B)$
$$\xymatrix{0\ar[r] & {\left[\begin{smallmatrix}X_1\\ Y_1 \end{smallmatrix}\right]_{\phi_1}}\ar[r]^{\left[\begin{smallmatrix}f_1\\ g_1 \end{smallmatrix}\right]} & {\left[\begin{smallmatrix}X_2\\ Y_2 \end{smallmatrix}\right]_{\phi_2}}\ar[r]^{\left[\begin{smallmatrix}f_2\\ g_2 \end{smallmatrix}\right]}  & {\left[\begin{smallmatrix}X_3\\ Y_3 \end{smallmatrix}\right]_{\phi_3}}\ar[r] & 0}$$
with $\left[\begin{smallmatrix} \alpha \\ \beta \end{smallmatrix}\right]\in\Hom_{\Lambda}(\left[\begin{smallmatrix}X_1\\ Y_1 \end{smallmatrix}\right]_{\phi_1}, \left[\begin{smallmatrix}W\\ V \end{smallmatrix}\right]_{\phi})$, we need looking for $\left[\begin{smallmatrix} \gamma \\ \delta \end{smallmatrix}\right]\in\Hom_{\Lambda}(\left[\begin{smallmatrix}X_2\\ Y_2 \end{smallmatrix}\right]_{\phi_2}, \left[\begin{smallmatrix}W\\ V \end{smallmatrix}\right]_{\phi})$ such that $\left[\begin{smallmatrix} \alpha \\ \beta \end{smallmatrix}\right]=\left[\begin{smallmatrix} \gamma \\ \delta \end{smallmatrix}\right]\left[\begin{smallmatrix}f_1\\ g_1 \end{smallmatrix}\right]$. Since $_BV$ is an injective module and $g_1: Y_1\to Y_2$ is monic, there is a $B$-map $\delta: Y_2\to V$ such that $\beta=\delta g_1$. Consider the $A$-map $\phi (1\otimes \delta): M\otimes_B Y_2\rightarrow W$.
Since $_AW$ is an injective module and $\phi_2: M\otimes_BY_2\to X_2$ is monic, there is an $A$-map $\gamma': X_2\to W$ such that $\phi(1\otimes{\delta})=\gamma'\phi_2$. Since $\alpha\phi_1=\phi(1\otimes{\beta})$ and $f_1\phi_1=\phi_2(1\otimes{g_1})$, we have $$\alpha\phi_1=\phi(1\otimes{\beta})=\phi(1\otimes{\delta})(1\otimes{g_1})=\gamma'\phi_2(1\otimes{g_1})=\gamma'f_1\phi_1.$$ So  $\alpha-\gamma'f_1=\eta\pi_1$ for some $\eta: \Cok(\phi_1)\to W$. Since $h_1: \Cok(\phi_1)\to\Cok(\phi_2)$ is monic and $W$ is injective, there is an $A$-map $\eta':\Cok(\phi_2)\to W$ such that $\eta=\eta'h_1$. We present the process as the diagram with exact rows and columns:
$$\xymatrix@R=3mm{ & 0\ar[d] &  & 0\ar[d] & 0\ar[d] &  \\
 0\ar[r] & M\otimes_BY_1\ar[rr]^{1\otimes{g_1}}\ar[dd]_{\phi_1}\ar[dr]_{1\otimes{\beta}} &  & M\otimes_BY_2\ar[r]^{1\otimes g_2}\ar[dd]^{\phi_2}\ar@{-->}[dl]^{1\otimes{\delta}} & M\otimes_BY_3\ar[r]\ar[dd]^{\phi_3} & 0 \\
   &  & M\otimes_BV\ar[dd]^<<<<{\phi} &  &  &  \\
 0\ar[r] & X_1\ar[rr]^<<<<<{f_1}\ar[dd]_{\pi_1}\ar[dr]_{\alpha} &  & X_2\ar[r]^{f_2}\ar[dd]^{\pi_2}\ar@{-->}[dl]^{\gamma'} & X_3\ar[r]\ar[dd]^{\pi_3} & 0 \\
   &  & W &  &  &  \\
 0\ar[r] & \Cok(\phi_1)\ar[rr]^{h_1}\ar[d]\ar@{-->}[ur]^{\eta} &  & \Cok(\phi_2)\ar[r]^{h_2}\ar[d]\ar@{-->}[ul]_{\eta'} & \Cok(\phi_3)\ar[r]\ar[d] & 0 \\
   & 0 &  & 0 & 0 &  }$$
Put $\gamma: =\gamma'+\eta'\pi_2\in\Hom_A(X_2, W)$. Then $$\gamma\phi_2=(\gamma'+\eta'\pi_2)\phi_2=\gamma'\phi_2=\phi(1\otimes{\delta})$$ and $$\gamma f_1=(\gamma'+\eta'\pi_2)f_1=\gamma'f_1+\eta'\pi_2f_1=\gamma'f_1+\eta'h_1\pi_1=\gamma'f_1+\eta\pi_1=\alpha.$$ This shows  $\left[\begin{smallmatrix} \gamma \\ \delta \end{smallmatrix}\right]\in\Hom_{\Lambda}(\left[\begin{smallmatrix}X_2\\ Y_2 \end{smallmatrix}\right]_{\phi_2}, \left[\begin{smallmatrix}W\\ V \end{smallmatrix}\right]_{\phi})$ and $\left[\begin{smallmatrix} \alpha \\ \beta \end{smallmatrix}\right]=\left[\begin{smallmatrix} \gamma \\ \delta \end{smallmatrix}\right]\left[\begin{smallmatrix}f_1\\ g_1 \end{smallmatrix}\right]$.

\vskip5pt

Next, we show that $\mathscr{S}(A, M, B)$ has enough injective objects. For $\left[\begin{smallmatrix}X\\ Y \end{smallmatrix}\right]_{\phi}\in\mathscr{S}(A, M, B)$, taking injective envelopes $\iota_Y: Y\to J$ and $\iota_C: \Cok(\phi)\to I$,
 we get exact sequences $0\rightarrow Y\stackrel {\iota_Y} \rightarrow J \stackrel {p_Y} \rightarrow \Cok(\iota_Y)\rightarrow 0$ and
 $0\rightarrow \Cok(\phi)\stackrel {\iota_C} \rightarrow I \stackrel  {p_C} \rightarrow \Cok(\iota_C) \rightarrow  0$. We take an injective envelope of the left $A$-module $M\otimes_BJ$ with inclusion $e: M\otimes_BJ\to E_J$.  Since $\phi: M\otimes_B Y\rightarrow X$ is monic and $E_J$ is an injective module, there is an  $A$-map $\alpha: X\to E_J$ satisfying $\alpha\phi=e(1\otimes\iota_Y)$, and hence we get a commutative diagram with exact rows:
$$\xymatrix@R=0.4cm{0\ar[r] & M\otimes_BY\ar[r]^-{\phi}\ar[d]^{1\otimes\iota_Y} & X\ar[r]^-{\pi}\ar[d]^{\alpha} & \Cok(\phi)\ar[r]\ar@{-->}[d]^{\beta} & 0 \\
0\ar[r] & M\otimes_BJ\ar[r]^-{e} & E_J\ar[r]^-{f} & \Cok(e)\ar[r] & 0}$$
Since $M_B$ is projective,  we have a commutative diagram with exact rows and columns:
$$\xymatrix@R=0.4cm{ & 0\ar[d] & 0\ar[d] & 0\ar[d] & \\
0\ar[r] & M\otimes_BY\ar[r]^{\phi}\ar[d]^{1\otimes\iota_Y} & X\ar[r]^-{\pi}\ar[d]^{\left[\begin{smallmatrix} \alpha \\ \iota_{_C}\pi\end{smallmatrix}\right]} & \Cok(\phi)\ar[r]\ar[d]^{\left[\begin{smallmatrix} \beta \\ \iota_{_C}\end{smallmatrix}\right]} & 0 \\
0\ar[r] & M\otimes_BJ\ar[r]_{\left[\begin{smallmatrix} e \\ 0 \end{smallmatrix}\right]}\ar[d]^{1\otimes p_{_Y}} & E_J\oplus I\ar[r]_-{\left[\begin{smallmatrix} f & 0 \\ 0 & 1\end{smallmatrix}\right]}\ar[d] & \Cok(e)\oplus I\ar[r]\ar[d] & 0\\
0\ar[r] & M\otimes_B\Cok(\iota_Y)\ar[r]^-{\psi}\ar[d] & L\ar[r]\ar[d] & C'\ar[r]\ar[d] & 0\\
 & 0 & 0 & 0 & }$$
(since $\iota_{_C}$ is monic, applying the Snake Lemma to the upper two rows we see that $\psi$ is monic). The left and middle columns give a short exact sequence in $\mathscr{S}(A, M, B)$
$$\xymatrix{0\ar[r] & {\left[\begin{smallmatrix} X\\ Y \end{smallmatrix}\right]_{\phi}}\ar[r] & {\left[\begin{smallmatrix}I\\ 0 \end{smallmatrix}\right]\oplus\left[\begin{smallmatrix}E_J\\ J \end{smallmatrix}\right]_e}\ar[r] & {\left[\begin{smallmatrix} L\\ \Cok(\iota_Y) \end{smallmatrix}\right]_{\psi}}\ar[r] & 0}. \quad \eqno(**)$$
This shows that $\mathscr{S}(A, M, B)$ has enough injective objects.

\vskip5pt

Finally, if $\left[\begin{smallmatrix}X\\ Y \end{smallmatrix}\right]_{\phi}$ is an indecomposable injective object of $\mathscr{S}(A, M, B)$, then $(**)$ splits.
By Lemma \ref{scat0}, $\mathscr{S}(A, M, B)$ is a Krull-Schimdt category. So
$\left[\begin{smallmatrix}X\\ Y \end{smallmatrix}\right]_{\phi}$  is either  of the form $\left[\begin{smallmatrix}I\\ 0 \end{smallmatrix}\right]$ or of the form $\left[\begin{smallmatrix}E_J\\ J \end{smallmatrix}\right]_{e}$.
\end{proof}

\begin{cor} Let $_AM_B$ be a bimodule with $M_B$ projective. Then $\mathscr{S}(A, M, B)$ is a {\rm Frobenius} category $($with the canonical exact structure$)$ if and only if both $A$ and $B$ are selfinjective algebras and $_AM$ is projective.
\label{corollary-frobenius-cat}
\end{cor}
\begin{proof} By Proposition \ref{enoughprojinjobj}, $\mathscr{S}(A,M,B)$ is a Frobenius category if and only if
$$\add(\left[\begin{smallmatrix} A \\ 0 \end{smallmatrix}\right]\oplus \left[\begin{smallmatrix} M \\ B \end{smallmatrix}\right]_{\rm Id})=\add(\left[\begin{smallmatrix} \D(A_A)\\ 0 \end{smallmatrix}\right]\oplus \left[\begin{smallmatrix} E_{\D(B_B)}\\ \D(B_B) \end{smallmatrix}\right]_{e}),$$
where $E_{\D(B_B)}$ is an injective envelope of the $A$-module $M\otimes_B\D(B_B)$ with embedding $e: M\otimes_B\D(B_B)\to E_{\D(B_B)}$.
Thus, if $\mathscr{S}(A,M,B)$ is {\rm Frobenius}, then $_AA$ is injective (thus ${\rm add}(\D(A_A)) = {\rm add}(_AA)$), $_AM$ is injective (thus $_AM\in
{\rm add}(\D(A_A)) = {\rm add}(_AA)$, hence $_AM$ is projective), and $_BB$ is injective.  Conversely, if $A$ and $B$ are selfinjective and $_AM$ is projective, then
${\rm add}(_BB) \cong {\rm add}(\D(B_B))$, hence $M\otimes_B\D(B_B)\in {\rm add}(M\otimes_B B) = {\rm add}(_AM)$, so $M\otimes_B\D(B_B)$ is a projective left $A$-module, and hence an injective left $A$-module.
Thus $M\otimes_B\D(B_B) = E_{\D(B_B)}.$ So ${\rm add}(\left[\begin{smallmatrix} E_{\D(B_B)}\\ \D(B_B) \end{smallmatrix}\right]_{e}) = {\rm add}(\left[\begin{smallmatrix} M\otimes_B\D(B_B)\\ \D(B_B) \end{smallmatrix}\right]_{\rm Id}) = \add(\left[\begin{smallmatrix} M\\ B \end{smallmatrix}\right]_{\rm Id})$, thus $\add(\left[\begin{smallmatrix} A \\ 0 \end{smallmatrix}\right]\oplus \left[\begin{smallmatrix} M \\ B \end{smallmatrix}\right]_{\rm Id})=\add(\left[\begin{smallmatrix} \D(A_A)\\ 0 \end{smallmatrix}\right]\oplus \left[\begin{smallmatrix} E_{\D(B_B)}\\ \D(B_B) \end{smallmatrix}\right]_{e}).$ Hence $\mathscr{S}(A,M,B)$ is a Frobenius category.
\end{proof}

\subsection{} To prove Theorem \ref{cotilting} we need some preparations.
A subcategory is a resolving subcategory of $\Lambda$-mod, if it contains all the projective $\Lambda$-modules and is closed under extensions, kernels of epimorphisms and direct summands ([AR]).

\begin{lem} \label{scat1} \ Let $_AM_B$ be a bimodule. Then $\mathscr{S}(A, M, B)$  is a resolving subcategory of $\modcat{\Lambda}$
if and only if $M_B$ is projective. \end{lem}
\begin{proof} By Lemma \ref{scat0} it suffices to prove that $\mathscr{S}(A, M, B)$ is closed under kernels of epimorphisms if and only if $M_B$ is projective.
Suppose that $M_B$ is projective. Let $f=\left[\begin{smallmatrix}
f_1 \\ f_2
\end{smallmatrix}\right]:\left[\begin{smallmatrix}
X_1\\ Y_1
\end{smallmatrix}\right]_{\phi_1}\rightarrow \left[\begin{smallmatrix}
X_2\\ Y_2
\end{smallmatrix}\right]_{\phi_2}$ be an epimorphism in $\modcat{\Lambda}$ with both $\left[\begin{smallmatrix}
X_1\\ Y_1
\end{smallmatrix}\right]_{\phi_1}$ and $\left[\begin{smallmatrix}
X_2\\ Y_2
\end{smallmatrix}\right]_{\phi_2}$ in $\mathscr{S}(A, M, B)$. So $X_1\stackrel{f_1}\rightarrow X_2$ and $Y_1\stackrel{f_2}\rightarrow Y_2$ are epic with $\Ker(f_1)\stackrel{i}\hookrightarrow X_1$ and  $\Ker(f_2)\stackrel{j}\hookrightarrow Y_1$. Since $M_B$ is projective,  we get the commutative diagram with exact rows:
$$\xymatrix @R=0.5cm{
0\ar[r] & M\otimes_B\Ker(f_2)\ar[r]^-{1\otimes j} &M\otimes_BY_1\ar[d]_{\phi_1}\ar[r]^{1\otimes f_2} &M\otimes_BY_2\ar[d]^{\phi_2}\ar[r] & 0 \\
0\ar[r] & \Ker(f_1)\ar[r]^{i} & X_1\ar[r]^{f_1} & X_2\ar[r] & 0.}$$
Thus there is a unique $A$-map $\phi: M\otimes_B\Ker(f_2)\rightarrow \Ker(f_1)$ such that $i\phi = \phi_1(1\otimes j)$. It is clearly that $\phi$ is monic, and hence $\Ker(f)=\left[\begin{smallmatrix}
\Ker(f_1)\\ \Ker(f_2)
\end{smallmatrix}\right]_\phi\in \mathscr{S}(A, M, B)$.

\vskip5pt

Conversely, suppose that $\mathscr{S}(A, M, B)$ is closed under kernels of epimorphisms. We claim that $M\otimes_B-$ is an exact functor.
In fact, let $0\rightarrow  K \stackrel {j} \rightarrow Y_1\stackrel f\rightarrow Y_2\rightarrow 0$ be an arbitrary exact sequence of $B$-modules.
Then $\left[\begin{smallmatrix}
1\otimes f \\ f
\end{smallmatrix}\right]:\left[\begin{smallmatrix}
M\otimes_BY_1\\ Y_1
\end{smallmatrix}\right]_{{\rm Id}}\lra \left[\begin{smallmatrix}
M\otimes_BY_2\\ Y_2
\end{smallmatrix}\right]_{{\rm Id}}$ is an epimorphism in $\modcat{\Lambda}$ with $\left[\begin{smallmatrix}
M\otimes_BY_i\\ Y_i
\end{smallmatrix}\right]_{{\rm Id}}\in\mathscr{S}(A, M, B)$ ($i = 1, 2$) and $\Ker \left[\begin{smallmatrix}
1\otimes f \\ f
\end{smallmatrix}\right] =\left[\begin{smallmatrix}
\Ker(1\otimes f)\\ K
\end{smallmatrix}\right]_\phi$, where $M\otimes_B K \stackrel{\phi}\longrightarrow \Ker(1\otimes f)$
is the unique $A$-map such that $\sigma\phi = 1\otimes j$, and $\sigma: \Ker(1\otimes f)\hookrightarrow M\otimes_BY_1$ is the embedding.
Since $\mathscr{S}(A, M, B)$ is closed under kernels of epimorphisms,  $\left[\begin{smallmatrix}
\Ker(1\otimes f)\\ K
\end{smallmatrix}\right]_\phi \in \mathscr{S}(A, M, B)$, i.e., $\phi$ is monic, and thus $1\otimes j$ is monic.
This proves the claim and hence $M_B$ is flat. Since $B$ is an Artin algebra and $M_B$ is finitely generated, it follows that $M_B$ is projective. \end{proof}

\begin{lem} \label{lemma1} {\rm ([XZ], \ Lemma \ 1.2)} \ For $X\in A\mbox{-}{\rm mod}$ and $Y\in B\mbox{-}{\rm mod}$, we have

\vskip5pt

$(1)$ \ $\Ext_{\Lambda}^1(\left[\begin{smallmatrix}0 \\ Y
\end{smallmatrix}\right], \left[\begin{smallmatrix} X \\ 0
\end{smallmatrix}\right])\cong\Hom_A(M\otimes_BY, X)$.

\vskip5pt

$(2)$ \ If $_AI$ is an injective $A$-module, then
$\Ext_{\Lambda}^{i+1}(\left[
\begin{smallmatrix}0\\Y\end{smallmatrix}\right], \left[
\begin{smallmatrix}I\\0\end{smallmatrix}\right])\cong\Ext_B^i(Y, \Hom_A(M,
I))$ for \ $i\geq0$.
\end{lem}

\begin{proof} For the convenience we include a short justification.

\vskip5pt

$(1)$ \ Let $0 \rightarrow K \stackrel i\rightarrow Q \stackrel p\rightarrow Y\rightarrow 0$ be an exact sequence with $Q$ a  projective left $B$-module.
Then $0
\rightarrow \left[\begin{smallmatrix} M\otimes Q
\\ K \end{smallmatrix}\right]_{1\otimes i}\overset{\left[\begin{smallmatrix}1
\\ i \end{smallmatrix}\right]}
\longrightarrow\left[\begin{smallmatrix} M\otimes Q\\
Q \end{smallmatrix}\right]_{\rm Id}\overset{\left[\begin{smallmatrix} 0\\ p
\end{smallmatrix}\right]} \longrightarrow\left[\begin{smallmatrix} 0\\ Y
\end{smallmatrix}\right]\rightarrow0$ is an exact sequence with $\left[\begin{smallmatrix} M\otimes Q\\
Q \end{smallmatrix}\right]_{\rm Id}$ a  projective left $\Lambda$-module. Applying $\Hom_\Lambda(-,\left[\begin{smallmatrix} X \\ 0
\end{smallmatrix}\right])$, since $\Hom_\Lambda(\left[\begin{smallmatrix} M\otimes Q\\
Q \end{smallmatrix}\right]_{\rm Id}, \left[\begin{smallmatrix} X \\ 0
\end{smallmatrix}\right])=0$,  we see \begin{align*}\Ext_{\Lambda}^1(\left[\begin{smallmatrix}0 \\ Y
\end{smallmatrix}\right], \left[\begin{smallmatrix} X \\ 0
\end{smallmatrix}\right])&\cong\Hom_A(\left[\begin{smallmatrix} M\otimes Q
\\ K \end{smallmatrix}\right]_{1\otimes i}, \left[\begin{smallmatrix} X \\ 0
\end{smallmatrix}\right]) = \{f\in \Hom_A(M\otimes_BQ, X) \ | \ f(1\otimes i) = 0\}\\ & \cong\Hom_A(M\otimes_BY, X).\end{align*}

\vskip5pt

$(2)$ \ If $i=0$ then the assertion follows from $(1)$ and the
Tensor-$\Hom$ adjunction isomorphism. Let $i\ge 1$. Using the abbreviation $(M, I) = \Hom_A(M, I)$, by the exact
sequence $0\rightarrow\left[\begin{smallmatrix} I\\ 0
\end{smallmatrix}\right] \rightarrow\left[\begin{smallmatrix} I\\ (M, I)
\end{smallmatrix}\right]_{\varphi}\rightarrow\left[\begin{smallmatrix}
0\\ (M, I)\end{smallmatrix}\right]\rightarrow0$, we get the exact sequence
$$\Ext_{\Lambda}^{i}(\left[\begin{smallmatrix}0\\Y\end{smallmatrix}\right],
\left[\begin{smallmatrix} I\\ (M,
I)\end{smallmatrix}\right]_{\varphi})
\longrightarrow\Ext_{\Lambda}^{i}(\left[\begin{smallmatrix}0\\Y\end{smallmatrix}\right],
\left[\begin{smallmatrix} 0\\ (M, I)\end{smallmatrix}\right])
\longrightarrow\Ext_{\Lambda}^{i+1}(\left[\begin{smallmatrix}0\\Y\end{smallmatrix}\right],
\left[\begin{smallmatrix}I\\0\end{smallmatrix}\right])
\longrightarrow\Ext_{\Lambda}^{i+1}(\left[\begin{smallmatrix}0\\Y\end{smallmatrix}\right],
\left[\begin{smallmatrix} I\\ (M,
I)\end{smallmatrix}\right]_{\varphi}).$$ Since
$\left[\begin{smallmatrix} I\\ (M, I)
\end{smallmatrix}\right]_{\varphi}$ is an injective
$\Lambda$-module, we have
$\Ext_{\Lambda}^{i+1}(\left[\begin{smallmatrix}0\\Y\end{smallmatrix}\right],
\left[\begin{smallmatrix}I\\0\end{smallmatrix}\right])\cong\Ext_{\Lambda}^{i}
(\left[\begin{smallmatrix}0\\Y\end{smallmatrix}\right],
\left(\begin{smallmatrix} 0\\ (M, I)\end{smallmatrix}\right])$, and then the assertion follows from $\Ext_{\Lambda}^{i}
(\left[\begin{smallmatrix}0\\Y\end{smallmatrix}\right],
\left[\begin{smallmatrix} 0\\ (M, I)\end{smallmatrix}\right])
\cong\Ext_B^i(Y, \Hom_A(M, I)),$ by using the injective resolution of $\Hom_A(M, I)$.
\end{proof}

\begin{lem} \label{scat2} If $M_B$ projective, then $\mathscr{S}(A, M, B) = \ ^\perp\left[\begin{smallmatrix}\D(A_A)\\ 0 \end{smallmatrix}\right]$.
\end{lem}
\begin{proof} We need to prove that
 $\left[\begin{smallmatrix}X\\ Y \end{smallmatrix}\right]_{\phi}\in\; ^\perp\left[\begin{smallmatrix}\D(A_A)\\ 0 \end{smallmatrix}\right]$ if and only if $\left[\begin{smallmatrix}X\\ Y \end{smallmatrix}\right]_{\phi}\in \mathscr{S}(A, M, B)$. Since $M_B$ is projective, $\Hom_A(M, \D(A_A))\cong \D M$ is an injective left $B$-module, and hence $\Ext_B^i(Y, \Hom_A(M, \D(A_A)))=0$ for all $i\geq1$. Applying $\Hom_{\Lambda}(-,
\left[\begin{smallmatrix}\D(A_A)\\0\end{smallmatrix}\right])$ to the exact
sequence $0 \rightarrow \left[\begin{smallmatrix} X
\\ 0 \end{smallmatrix}\right]
\rightarrow\left[\begin{smallmatrix} X\\
Y \end{smallmatrix}\right]_{\phi}
\rightarrow\left[\begin{smallmatrix} 0\\ Y
\end{smallmatrix}\right]\rightarrow0$, by Lemma \ref{lemma1} we get the
commutative diagram with the upper row being exact
$$\xymatrix @R=0.5cm @C=0.8cm {{\Hom_\Lambda(\left[\begin{smallmatrix}X\\0\end{smallmatrix}\right],
\left[\begin{smallmatrix}\D(A_A)\\0\end{smallmatrix}\right])}
\ar[r]\ar[d]^{\cong}&{\Ext_{\Lambda}^1(\left[\begin{smallmatrix}0\\Y\end{smallmatrix}\right],
\left[\begin{smallmatrix}\D(A_A)\\0\end{smallmatrix}\right])}
\ar[r]\ar[d]^{\cong}&{\Ext_{\Lambda}^1(\left[\begin{smallmatrix}X\\Y\end{smallmatrix}\right]_{\phi},
\left[\begin{smallmatrix}\D(A_A)\\0\end{smallmatrix}\right])}\\
{\Hom_A(X, \D(A_A))}\ar[r]^-{\phi^*}& {\Hom_A(M\otimes_BY, \D(A_A)),}}$$ and
the following exact sequence for $i\geq1$
$$\xymatrix @R=0.5cm @C=0.3cm
{{\Ext^i_\Lambda(\left[\begin{smallmatrix} X\\Y\end{smallmatrix}\right]_{\phi}, \left[\begin{smallmatrix}\D A\\0\end{smallmatrix}\right])} \ar[r]
& {\Ext_\Lambda^i(\left[\begin{smallmatrix}X\\0\end{smallmatrix}\right], \left[\begin{smallmatrix}\D A\\0\end{smallmatrix}\right])} \ar[r]\ar[d]^{\cong}
& {\Ext_\Lambda^{i+1}(\left[\begin{smallmatrix}0\\Y\end{smallmatrix}\right], \left[\begin{smallmatrix}\D A\\0\end{smallmatrix}\right])}\ar[r]\ar[d]^-{\cong}
& {\Ext_\Lambda^{i+1}(\left[\begin{smallmatrix}X\\Y\end{smallmatrix}\right]_{\phi}, \left[\begin{smallmatrix} \D A\\0\end{smallmatrix}\right])} \ar[r] & \cdots
\\ & 0= \Ext_A^i(X, \D A)  & {\Ext_B^i(Y, \Hom(M, \D A)) = 0.}}$$
So $\left[\begin{smallmatrix}X\\Y\end{smallmatrix}\right]_{\phi} \in \ ^\perp\left[\begin{smallmatrix}\D(A_A)\\0\end{smallmatrix}\right]$ if and
only if $\phi^*: \Hom_A(X, \D(A_A))\rightarrow\Hom_A(M\otimes_BY,
\D(A_A))$ is an epimorphism, and if and only if $\phi: M\otimes_B Y \rightarrow X$ is monic.
\end{proof}

\subsection{Proof of Theorem \ref{cotilting}} $(1)$ \ By Lemma \ref{scat1} we have the implications ${\rm (i)} \Longleftrightarrow {\rm (ii)}$.

\vskip5pt

${\rm (iii)} \Longrightarrow {\rm (ii)}:$ \ Since $\mathscr{S}(A, M, B)={}^\perp T$, it is clear that $\mathscr{S}(A, M, B)$ is a resolving subcategory of
$\Lambda$-mod.

\vskip5pt

${\rm (i)} \Longrightarrow {\rm (iii)}:$ \ Since $M_B$ is projective, $\Hom_A(M, \D(A_A))\cong \D M$ is an injective left $B$-module, and hence $\left[\begin{smallmatrix} 0 \\ \Hom_{A}(M, \D(A_A)) \end{smallmatrix}\right]$ is an injective left $\Lambda$-module. By the exact sequence in $\Lambda$-mod
$$\xymatrix{0\ar[r] & {\left[\begin{smallmatrix}\D(A_A)\\ 0 \end{smallmatrix}\right]}\ar[r] & {\left[\begin{smallmatrix}\D(A_A)\\ \Hom_{A}(M, \D(A_A)) \end{smallmatrix}\right]_\varphi}\ar[r] & {\left[\begin{smallmatrix} 0 \\ \Hom_{A}(M, \D(A_A)) \end{smallmatrix}\right]}\ar[r] & 0 }$$
we see that $\injd_{\Lambda}\left[\begin{smallmatrix}\D(A_A)\\ 0 \end{smallmatrix}\right]\le 1$.

\vskip5pt

Let $\alpha: \D(B_B)\rightarrow \Hom_{A}(M, E_{\D(B_B)})$ be  the image of $e\in \Hom_A(M\otimes_B\D(B_B), E_{\D(B_B)})$ under the adjunction isomorphism
$$\Hom_A(M\otimes_B\D(B_B), E_{\D(B_B)})\cong \Hom_B(\D(B_B), \Hom_{A}(M, E_{\D(B_B)})).$$
By the naturalness of the adjunction isomorphisms we have the commutative diagram
$$\xymatrix{\Hom_A(M\otimes_B\Hom_A(M, E_{\D(B_B)}), E_{\D(B_B)})\ar[r]^-{\cong}
\ar[d]^-{({\rm Id}\otimes_B\alpha, E_{\D(B_B)})} &
\Hom_B(\Hom_A(M, E_{\D(B_B)}), \Hom_A(M, E_{\D(B_B)}))\ar[d]^-{(\alpha, \ (M, E_{\D(B_B)}))} \\
\Hom_A(M\otimes_B\D(B_B), E_{\D(B_B)})\ar[r]^-{\cong} & \Hom_B(\D(B_B), \Hom_A(M, E_{\D(B_B)})).}$$
Let $\varphi: \Hom_A(M\otimes_B \Hom_A(M, E_{\D (B_B)}), \D (B_B))$ be the involution map. By the above commutative diagram
we can get $\varphi (1\otimes_B \alpha) = e$. So we get a $\Lambda$-map $\left[\begin{smallmatrix} 1 \\ \alpha \end{smallmatrix}\right]: \left[\begin{smallmatrix}E_{\D(B_B)}\\ \D(B_B) \end{smallmatrix}\right]_{e}\longrightarrow\left[\begin{smallmatrix}E_{\D(B_B)}\\ \Hom_A(M, E_{\D(B_B)}) \end{smallmatrix}\right]_\varphi$, and we have the exact sequence in $\Lambda$-mod
$$\xymatrix@C=1.2cm{0\ar[r] & {\left[\begin{smallmatrix}E_{\D(B_B)}\\ \D(B_B) \end{smallmatrix}\right]_{e}}\ar[r]^-{\left(\begin{smallmatrix} {\left[\begin{smallmatrix} 1 \\ \alpha \end{smallmatrix}\right]} \\ {\left[\begin{smallmatrix} 0 \\ 1 \end{smallmatrix}\right]} \end{smallmatrix}\right)} & {\left[\begin{smallmatrix}E_{\D(B_B)}\\ \Hom_A(M, E_{\D(B_B)}) \end{smallmatrix}\right]_\varphi\oplus\left[\begin{smallmatrix} 0 \\ \D(B_B) \end{smallmatrix}\right]}\ar[r]^-{(\left[\begin{smallmatrix}0\\ 1 \end{smallmatrix}\right], \left[\begin{smallmatrix}0\\ -\alpha \end{smallmatrix}\right])} & {\left[\begin{smallmatrix} 0 \\ \Hom_{A}(M, E_{\D(B_B)}) \end{smallmatrix}\right]}\ar[r] & 0.}$$
Since $E_{\D(B_B)}$ is an injective left $A$-module, it follows that $\Hom_{A}(M, E_{\D(B_B)})\in {\rm add}(\Hom_A(M, \D(A_A)))$, so $\Hom_{A}(M, E_{\D(B_B)})$ is an injective left $B$-module, and hence $\left[\begin{smallmatrix} 0 \\ \Hom_{A}(M, E_{\D(B_B)}) \end{smallmatrix}\right]$ is an injective $\Lambda$-module. This shows $\injd_{\Lambda}\left[\begin{smallmatrix}E_{\D(B_B)}\\ \D(B_B) \end{smallmatrix}\right]_{e}\le 1$.
Thus  $\injd_{\Lambda}T = \injd_{\Lambda}(\left[\begin{smallmatrix}\D(A_A)\\ 0 \end{smallmatrix}\right]\oplus\left[\begin{smallmatrix}E_{\D(B_B)}\\ \D(B_B) \end{smallmatrix}\right]_{e})\le 1$.

\vskip5pt

By Proposition \ref{enoughprojinjobj}$(2)$, $T$ is an injective object of $\mathscr{S}(A, M, B)$, so $\Ext_{\Lambda}^1(T, T)=0$. It is clear that the number of the pairwise non-isomorphic indecomposable direct summands of $T$ is the number of the simple $\Lambda$-modules. So $_\Lambda T$ is a cotilting $\Lambda$-module.

\vskip5pt

Since $T$ is an injective object of $\mathscr{S}(A, M, B)$ and $\injd_{\Lambda}T\le 1$,  we have $\mathscr{S}(A, M, B)\subseteq{}^\perp T$.
By Lemma \ref{scat2}, $\mathscr{S}(A, M, B)= \ ^\perp\left[\begin{smallmatrix}D(A_A)\\ 0 \end{smallmatrix}\right]\supseteq \ ^\perp T$. Thus $\mathscr{S}(A, M, B)={}^\perp T$.

 \vskip5pt

If there is another cotilting $\Lambda$-module $L$ such that
$\mathscr{S}(A, M, B)= {}^\perp L$. Then $T\oplus L$
is also a cotilting $\Lambda$-module. By comparing the number of pairwise non-isomorphic indecomposable direct summands of $T\oplus L$ and $T$, we see the uniqueness of $T$, up to the multiplicities of indecomposable direc summands.

\vskip5pt

$(2)$ \ The following construction is from C. Ringel and M. Schmidmeier [RS2]. Let $\left[\begin{smallmatrix}
X\\ Y
\end{smallmatrix}\right]_{\phi}\in\Lambda$-mod. Define ${\rm Mimo}(\phi)$ to be the $\Lambda$-module $\left[\begin{smallmatrix}
X\oplus\ik(\phi)\\ Y
\end{smallmatrix}\right]_{\left[\begin{smallmatrix}
\phi\\ e
\end{smallmatrix}\right]}$, where $e: M\otimes_B Y\lra\ik(\phi)$ is an extension of the injective envelope ${\rm Ker}(\phi)\hookrightarrow\ik(\phi)$. Then it is clear that ${\rm Mimo}(\phi)$ is well-defined (i.e., independent of the choice of $e$) and it is in $\mathscr{S}(A, M, B)$. For any  $\left[\begin{smallmatrix}
X\\ Y
\end{smallmatrix}\right]_{\phi}\in\Lambda$-mod, by the similar argument as in [RS2, Prop. 2.4], one can see that $\left[\begin{smallmatrix} (1, 0)\\ 1 \end{smallmatrix}\right]: {\rm Mimo}(\phi)=\left[\begin{smallmatrix} X\oplus\ik(\phi)\\ Y \end{smallmatrix}\right]_{\left[\begin{smallmatrix} \phi\\ e \end{smallmatrix}\right]}\lra\left[\begin{smallmatrix}X\\ Y \end{smallmatrix}\right]_{\phi}$ is a minimal right $\mathscr{S}(A, M, B)$-approximation of $\left[\begin{smallmatrix}
X\\ Y
\end{smallmatrix}\right]_{\phi}$.  Thus $\mathscr{S}(A, M, B)$ is a contravariantly finite subcategory of $\Lambda$-mod.

\vskip5pt

By [KS, Corol. 0.3], a resolving contravariantly
finite subcategory of $\Lambda$-mod is functorially finite, and by [AS, Thm. 2.4], an extension-closed
functorially finite subcategory of $\Lambda$-mod has Auslander-Reiten
sequences. Thus, if $M_B$ is projective, then by $(1)$,
$\mathscr{S}(A, M, B)$ is functorially finite in
$\Lambda$-mod, and hence $\mathscr{S}(A, M, B)$ has Auslander-Reiten sequences.
$\hfill \square$

\subsection{} Recall that each right $\Lambda$-module is identified with a triple $(X, Y)_{\phi}$, where $X\in {\rm mod}A$, $Y\in {\rm mod}B$, and $\phi: X\otimes_A M\rightarrow Y$ is a right $B$-map; and a right $\Lambda$-map is identified with a pair $(f_1, f_2): (X_1, Y_1)_{\phi_1}\rightarrow (X_2, Y_2)_{\phi_2}$, where $f_1: X_1\rightarrow X_2$ is an $A$-map and $f_2: Y_1\rightarrow Y_2$ a $B$-map, such that $f_2\phi_1 = \phi_2(f_1\otimes 1)$.
The injective right $\Lambda$-modules are exactly
$(I, 0)$ and $(\Hom_B(M, J), J)_\varphi$, where $I$ (resp. $J$) runs over the injective right $A$-modules (resp. $B$-modules), and $\varphi:
{\rm Hom}_B(M, J)\otimes_AM  \rightarrow J$ is the involution map given
by $\varphi (f\otimes m)=f(m).$ See [ARS, p.73].

\vskip5pt

All the results obtained so far have the right module versions. We only write down what is needed later.
The right module version of $\mathscr S(A, M, B)$ is $\mathscr S(A, M, B)_r$, which is the subcategory of mod$\Lambda$ consisting of the triple $(U, V)_\phi$, where $X\in {\rm mod}A$, $Y\in {\rm mod}B$, and $\phi: X\otimes_A M \rightarrow Y$ is a monic right $B$-map. Then $\mathscr S(A, M, B)_r$ is a resolving subcategory of $\Lambda$-{\rm mod} if and only if $_AM$ projective, and if and only if $\mathscr S(A, M, B)_r = \ ^\perp(\D(_AA),0)$. The following result is only a part of the right module version of Theorem \ref{cotilting}, which is what we will need later.

\vskip5pt

{\bf A right module version of Theorem \ref{cotilting}.} {\it Let $_AM_B$ be a bimodule with $_AM$ projective. Then $U_\Lambda :=(\D(_AA), E_{\D(_AA)})_{e}\oplus (0, \D(_BB))$ is a unique cotilting right $\Lambda$-module, up to multiplicities of indecomposable direct summands, such that $\mathscr{S}(A, M, B)_r ={}^\perp (U_\Lambda)$, where  $E_{\D(_AA)}$ is an injective envelope of  $\D(_AA)\otimes_AM$ with embedding $e: \D(_AA)\otimes_AM\hookrightarrow E_{\D(_AA)}$.

\vskip5pt

In particular, if $\D(_AA)\otimes_AM$ is an injective right $B$-module, then $U_\Lambda =(\D(_AA), \D(_AA)\otimes_AM)_{\rm Id}\oplus (0, \D(_BB))$ is a unique cotilting right $\Lambda$-module, up to multiplicities of indecomposable direct summands, such that $\mathscr{S}(A, M, B)_r = {}^\perp (U_\Lambda)$.}

\section{\bf Proof of Theorem \ref{theorem-recollement-additive-cat}}

\subsection{} Let $\mathscr{A}, \mathscr{B}, \mathscr{C}$ be additive categories. The diagram of functors
\begin{center}
\begin{picture}(100,40)
\put(13,20){\makebox(-22,2) {$\mathscr B$}}
\put(37,20){\makebox(25,0.8) {$\mathscr A$}}
\put(88,20){\makebox(25,0.5){$\mathscr C$}}
\put(10,20){\vector(1,0){30}} \put(60,20){\vector(1,0){30}}
\put(25,23){\makebox(3,1){\scriptsize$i_\ast$}}
\put(74,23){\makebox(3,1){\scriptsize$j^\ast$}}
\put(40,11){\vector(-1,0){30}}
\put(90,11){\vector(-1,0){30}}
\put(25,15){\makebox(3,1){\scriptsize$i^{!}$}}
\put(74,15){\makebox(3,1){\scriptsize$j_{*}$}}
\put(40,28){\vector(-1,0){30}} \put(90,28){\vector(-1,0){30}}
\put(25,32){\makebox(3,1){\scriptsize$i^{*}$}}
\put(74,32){\makebox(3,1){\scriptsize$j_{!}$}}
\end{picture}
\end{center}
\vskip-10pt \noindent is a recollement of $\mathscr{A}$ relative to $\mathscr{B}$ and $\mathscr{C}$, if the conditions $({\rm R1}), ({\rm R2}), ({\rm R3})$ are satisfied$:$

\vskip5pt

$({\rm R1})$  \ \ $(i^*, i_*)$, $(i_*, i^!)$, $(j_!, j^*)$ and $(j^*, j_*)$ are adjoint pairs$;$

$({\rm R2})$ \ \ $i_*$, $j_!$ and $j_*$ are fully faithful$;$

$({\rm R3})$ \ \ $\Ima i_*=\Ker j^*$.

\vskip5pt

Since the functors in an adjoint pair between additive categories are additive functors,
all the six functors in a recollement of additive categories are additive.

\vskip5pt

\subsection{} The following fact is easy.

\begin{lem} \ Let $\mathscr{A}$ and $\mathscr{B}$ be additive categories with subcategories $\mathscr{X}$ and $\mathscr{Y}$, respectively,
$(F, G)$ an adjoint pair with $F: \mathscr{A}\to\mathscr{B}$ and $G: \mathscr{B}\to\mathscr{A}$. If $F\mathscr{X}\subseteq\mathscr{Y}$ and $G \mathscr{Y}\subseteq\mathscr{X}$, then there is an induced adjoint pair $(\overline{F}, \overline{G})$ with $\overline{F}: \mathscr{A}/\mathscr{X}\to\mathscr{B}/\mathscr{Y}$ and
 $\overline{G}: \mathscr{B}/\mathscr{Y}\to\mathscr{A}/\mathscr{X}$.
\label{lemma-induction-quotient-cat}
\end{lem}

Let $\Lambda=\left[\begin{smallmatrix}
A&M\\0&B
\end{smallmatrix}\right]$ be an Artin algebra. Define five functors as follows.
$$\begin{array}{ll}
i^*: \mathscr{S}(A, M, B)\to\modcat{A}, \ \left[\begin{smallmatrix} X\\
Y \end{smallmatrix}\right]_{\phi}\mapsto\Cok(\phi); & j_!: \modcat{B}\to\mathscr{S}(A, M, B), \ Y\mapsto \left[\begin{smallmatrix} M\otimes_BY \\
Y \end{smallmatrix}\right]_{\rm Id};\\
i_*: \modcat{A}\to\mathscr{S}(A, M, B), \ X\mapsto \left[\begin{smallmatrix} X\\
0 \end{smallmatrix}\right]; &  j^*: \mathscr{S}(A, M, B)\to \modcat{B}, \ \left[\begin{smallmatrix} X\\
Y \end{smallmatrix}\right]_{\phi}\mapsto Y;\\
i_!: \mathscr{S}(A, M, B)\to\modcat{A}, \ \left[\begin{smallmatrix} X\\
Y \end{smallmatrix}\right]_{\phi}\mapsto X. &
\end{array}$$

\begin{lem}
Let $\Lambda=\left[\begin{smallmatrix}
A&M\\0&B
\end{smallmatrix}\right]$ be an Artin algebra. Then

\vskip5pt

$(1)$ \  $(i^*, i_*)$, $(i_*, i^!)$, and $(j_!, j^*)$ are adjoint pairs$;$

\vskip5pt

$(2)$  \ $i_*$ and $j_!$ are fully faithful$;$

\vskip5pt

$(3)$ \ $\Ima i_*=\Ker j^*$.
\label{lemma-adjoint}
\end{lem}

\begin{proof}
$(1)$ \ For any $\left[\begin{smallmatrix} X\\
Y \end{smallmatrix}\right]_{\phi}\in\mathscr{S}(A, M, B)$, $W\in\modcat{A}$ and $V\in\modcat{B}$, we have the isomorphisms:
$$\begin{array}{l}
\Hom_A(i^*(\left[\begin{smallmatrix} X\\
Y \end{smallmatrix}\right]_{\phi}), W)=\Hom_A(\Cok(\phi), W)\cong\Hom_{\Lambda}(\left[\begin{smallmatrix} X\\
Y \end{smallmatrix}\right]_{\phi}, \left[\begin{smallmatrix} W\\
0 \end{smallmatrix}\right])=\Hom_{\Lambda}(\left[\begin{smallmatrix} X\\
Y \end{smallmatrix}\right]_{\phi}, i_*(W));\\
\Hom_{\Lambda}(i_*(W), \left[\begin{smallmatrix} X\\
Y \end{smallmatrix}\right]_{\phi})=\Hom_{\Lambda}(\left[\begin{smallmatrix} W\\
0 \end{smallmatrix}\right], \left[\begin{smallmatrix} X\\
Y \end{smallmatrix}\right]_{\phi})\cong\Hom_A(W, X)=\Hom_A(W, i_!(\left[\begin{smallmatrix} X\\
Y \end{smallmatrix}\right]_{\phi}));\\
\Hom_{\Lambda}(j_!(V), \left[\begin{smallmatrix} X\\
Y \end{smallmatrix}\right]_{\phi})=\Hom_{\Lambda}(\left[\begin{smallmatrix} M\otimes_BV\\
V \end{smallmatrix}\right]_1, \left[\begin{smallmatrix} X\\
Y \end{smallmatrix}\right]_{\phi})\cong\Hom_B(V, Y)=\Hom_B(V, j^*(\left[\begin{smallmatrix} X\\
Y \end{smallmatrix}\right]_{\phi})).
\end{array}$$
These show that $(i^*, i_*)$, $(i_*, i^!)$, and $(j_!, j^*)$ are adjoint pairs.

\vskip5pt

$(2)$ \ For any $X_1, X_2\in\modcat{A}$ and $Y_1, Y_2\in\modcat{B}$, we have the isomorphisms:
$$\begin{array}{l}
\Hom_{\Lambda}(i_*(X_1), i_*(X_2))=\Hom_{\Lambda}(\left[\begin{smallmatrix} X_1\\
0 \end{smallmatrix}\right], \left[\begin{smallmatrix} X_2\\
0 \end{smallmatrix}\right])\cong\Hom_A(X_1, X_2); \\
\Hom_{\Lambda}(j_!(Y_1), j_!(Y_2))=\Hom_{\Lambda}(\left[\begin{smallmatrix} M\otimes_BY_1\\
Y_1 \end{smallmatrix}\right]_{\rm Id},\left[\begin{smallmatrix} M\otimes_BY_2\\
Y_2 \end{smallmatrix}\right]_{\rm Id})\cong\Hom_B(Y_1, Y_2).
\end{array}$$
These show that $i_*$ and $j_!$ are fully faithful.

\vskip5pt

$(3)$ \ This is clear.
\end{proof}

Suppose that $_AM_B$ satisfies the condition ${\rm (IP)}$. By Proposition \ref{enoughprojinjobj}, the injective objects of $\mathscr{S}(A, M, B)$
are exactly $\left[\begin{smallmatrix} I \\ 0 \end{smallmatrix}\right]$ and $\left[\begin{smallmatrix}M\otimes_BJ\\ J \end{smallmatrix}\right]_1$,
where $I$ runs over injective $A$-modules,  and $J$ runs over injective $B$-modules.
Thus, by the constructions all the functors $i^*$, $i_*$, $i_!$, $j_!$ and $j^*$ preserve injective objects. It follows from Lemmas \ref{lemma-induction-quotient-cat} and \ref{lemma-adjoint} that there are the induced functors:
$$\begin{array}{ll}
\overline{i^*}: \overline{\mathscr{S}(A, M, B)}\to A\mbox{-}\overline{\rm mod}, \quad \left[\begin{smallmatrix} X\\
Y \end{smallmatrix}\right]_{\phi}\mapsto\Cok(\phi); & \overline{j_!}: B\mbox{-}\overline{\rm mod}\to\overline{\mathscr{S}(A, M, B)}, \quad Y\mapsto \left[\begin{smallmatrix} M\otimes_BY \\
Y \end{smallmatrix}\right]_{\rm Id};\\
\overline{i_*}: A\mbox{-}\overline{\rm mod}\to\overline{\mathscr{S}(A, M, B)}, \quad X\mapsto \left[\begin{smallmatrix} X\\
0 \end{smallmatrix}\right]; &  \overline{j^*}: \overline{\mathscr{S}(A, M, B)}\to B\mbox{-}\overline{\rm mod}, \quad \left[\begin{smallmatrix} X\\
Y \end{smallmatrix}\right]_{\phi}\mapsto Y;\\
\overline{i_!}: \overline{\mathscr{S}(A, M, B)}\to A\mbox{-}\overline{\rm mod}, \quad \left[\begin{smallmatrix} X\\
Y \end{smallmatrix}\right]_{\phi}\mapsto X &
\end{array}$$
such that the following fact holds:
\begin{lem} \label{corollary-adjoint-pairs} \ If $_AM_B$ satisfies the condition ${\rm (IP)}$, then $(\overline{i^*},\ \overline{i_*})$, $(\overline{i_*},\ \overline{i^!})$ and $(\overline{j_!},\ \overline{j^*})$ are adjoint pairs.

\vskip5pt

Moreover, if in addition both $A$ and $B$ are selfinjective algebras, then all the functors $\overline{i^*}, \ \overline{i_*}, \ \overline{i^!}, \ \overline{j_!}, \ \overline{j^*}$
are triangle functors between triangulated categories.
\end{lem}
\begin{proof} We only need to justify the last assertion. In this case, both $A\mbox{-}\overline{\rm mod}$ and $B\mbox{-}\overline{\rm mod}$ are triangulated categories.  By Example \ref{mainexample}$(7)$, $_AM$ is projective; and then by Corollary \ref{corollary-frobenius-cat}, $\mathscr{S}(A,M,B)$ is a {\rm Frobenius} category, hence $\overline{\mathscr{S}(A, M, B)}$ is also a triangulated category. See [H1, p.16]. Recall the distinguished triangles in the stable category of a Frobenius category.
Each exact sequence $0
\rightarrow X_1 \stackrel{u}{\rightarrow} X_2
\stackrel{v}{\rightarrow} X_3 \rightarrow 0$ in $A$-mod
gives rise to a distinguished triangle $X_1
\stackrel{{u}}{\rightarrow} X_2 \stackrel{{v}}{\rightarrow} X_3
{\rightarrow} X[1]$ \ of $A\mbox{-}\overline{\rm mod}$, and conversely, each
distinguished triangle of $A\mbox{-}\overline{\rm mod}$ is of this form
up to an isomorphism of triangles (see [H1], Chap. 1, Sect. 2; also [CZ], Lemma 1.2). Since the functor
$i_*: A\mbox{-}{\rm mod}\to \mathscr{S}(A, M, B)$ given by $X\mapsto \left[\begin{smallmatrix} X\\ 0 \end{smallmatrix}\right]$ preserves exact sequences, $\overline{i_*}: A\mbox{-}\overline{\rm mod}\to\overline{\mathscr{S}(A, M, B)}$ preserves the distinguished triangles, i.e., $\overline{i_*}$ is a triangle functor. Note that in an adjoint pair $(F, G)$ between triangulated categories, $F$ is
a triangle functor if and only if so is $G$ ([Ke, 6.7], [N, p.179]). Thus,  $\overline{i^*}$ and $\overline{i_!}$ are triangle functors.

\vskip5pt

Similarly, $\overline{j^*}: \overline{\mathscr{S}(A, M, B)}\to B\mbox{-}\overline{\rm mod}$
is a triangle functor, and then so is $\overline{j_!}$.
\end{proof}

The following lemma is crucial in the proof of Theorem \ref{theorem-recollement-additive-cat}.

\begin{lem} Assume that $_AM_B$ satisfies the condition ${\rm (IP)}$. Then there exists a fully faithful functor $\overline{j_*}: B\mbox{-}\overline{\rm mod}\to\overline{\mathscr{S}(A, M, B)}$ such that $(\overline{j^*}, \overline{j_*})$ is an adjoint pair.
\label{lemma-fully-faithfu-j}
\end{lem}

\begin{proof}  The following construction is  similar as [Z2, Thm. 3.5]. Define a functor $j_*: \modcat{B}\to\overline{\mathscr{S}(A, M, B)}$ as follows.
For $Y\in\modcat{B}$, define $j_*(Y): =\left[\begin{smallmatrix} E_Y\\
Y \end{smallmatrix}\right]_{\psi},$ where $E_Y$ is an injective envelope of the $A$-module $M\otimes_BY$ with embedding $\psi: M\otimes_BY\to E_Y$.
(This is clearly well-defined. One can see this also from the argument below, by taking $h = {\rm Id}_Y$.)
For $h: Y\to Y'$, there is a commutative diagram with exact rows
$$\xymatrix@R=0.5cm{0\ar[r] & M\otimes_BY\ar[r]^-{\psi}\ar[d]^{1\otimes h} & E_Y\ar[r]^-{\pi}\ar@{-->}[d]^{f} & \Cok(\psi)\ar[r]\ar@{-->}[d] & 0\\
0\ar[r] & M\otimes_BY'\ar[r]^-{\psi'} & E_{Y'}\ar[r] & \Cok(\psi')\ar[r] & 0}$$
and we define $j_*(h): = \overline{\left[\begin{smallmatrix} f\\
h\end{smallmatrix}\right]}: \ \left[\begin{smallmatrix} E_Y\\
Y \end{smallmatrix}\right]_{\psi}\to \left[\begin{smallmatrix} E_{Y'}\\
Y' \end{smallmatrix}\right]_{\psi'}$. We claim that $j_*(h)$ is well-defined, and hence the functor $j_*: \modcat{B}\to\overline{\mathscr{S}(A, M, B)}$ is well-defined.

\vskip5pt

In fact, if there is another $A$-map $f': E_Y\to E_{Y'}$ such that $f'\psi=\psi'(1\otimes h)$,
then $f'\psi=f\psi$, i.e.,  $(f-f')\psi=0$ and hence $f-f': E_Y\to E_{Y'}$ factors through $\Cok(\psi)$. Furthermore, $f-f'$ factors through an injective envelope $I$ of $\Cok(\psi)$. Then it is clear that $\left[\begin{smallmatrix} f\\
h \end{smallmatrix}\right]-\left[\begin{smallmatrix} f'\\
h \end{smallmatrix}\right] = \left[\begin{smallmatrix} f-f' \\
0 \end{smallmatrix}\right]: \left[\begin{smallmatrix} E_Y\\
Y \end{smallmatrix}\right]_{\psi}\to \left[\begin{smallmatrix} E_{Y'}\\
Y' \end{smallmatrix}\right]_{\psi'}$ factors through the injective object $\left[\begin{smallmatrix} I\\
0 \end{smallmatrix}\right]$ of $\mathscr{S}(A, M, B)$. So $\overline{\left[\begin{smallmatrix} f\\
h \end{smallmatrix}\right]}=\overline{\left[\begin{smallmatrix} f'\\
h \end{smallmatrix}\right]}$ in $\overline{\mathscr{S}(A, M, B)}$.

\vskip5pt

Next, we claim that the functor $j_*: \modcat{B}\to\overline{\mathscr{S}(A, M, B)}$ induces a functor $\overline{j_*}: B\mbox{-}\overline{\rm mod}\to\overline{\mathscr{S}(A, M, B)}$.
For this, assume that $h: Y\to Y'$ factors through an injective $B$-module $J$ via $h=h_2h_1$ with some $h_1\in\Hom_B(Y, J)$ and some $h_2\in\Hom_B(J, Y')$.
Taking an injective envelope $E_J$ of the $A$-module $M\otimes_BJ$ with embedding $\eta: M\otimes_BJ\to E_J$, then there are
$\sigma\in\Hom_A(E_Y, E_J)$ and $\delta\in\Hom_A(E_J, E_{Y'})$,
such that $\eta(1\otimes h_1)=\sigma\psi$ and $\psi'(1\otimes h_2)=\delta\eta$. Thus $f\psi=\psi'(1\otimes h)=\psi'(1\otimes h_2)(1\otimes h_1)=\delta\eta(1\otimes h_1)=\delta\sigma\psi$, i.e., $(f-\delta\sigma)\psi=0$. Hence $f-\delta\sigma$ factors through $\Cok(\psi)$ via $f-\delta\sigma=\gamma\pi$ with some $\gamma\in\Hom_A(\Cok(\psi), E_{Y'})$. Consider an injective envelope $I$ of $\Cok(\psi)$ with embedding $\alpha:\Cok(\psi)\to I$. Then there is a $\beta\in\Hom_A(I, E_{Y'})$ such that $\gamma=\beta\alpha$, and then we have
$f=\delta\sigma+\beta\alpha\pi$. We present this process as the diagram:
$$\xymatrix@R=0.6cm{0\ar[r] & M\otimes_BY\ar[rrr]^-{\psi}\ar[dd]_{1\otimes h}\ar[dr]^{1\otimes h_1} & & & E_Y\ar[r]^-{\pi}\ar[dd]^-{f}\ar@{-->}[dl]_{\sigma} & \Cok(\psi)\ar[rr]\ar[dd]\ar[dr]^{\alpha}\ar@{-->}[ddl]_-{\gamma} & & 0\\ & & M\otimes_BJ\ar[r]^-{\eta}\ar[dl]_-{1\otimes h_2} & E_J\ar@{-->}[dr]^-{\delta} & & & I\ar@{-->}[dll]_<<<<<<<<<{\beta} & \\
0\ar[r] & M\otimes_BY'\ar[rrr]^-{\psi'} & & & E_{Y'}\ar[r] & \Cok(\psi')\ar[rr] & &0 \ .}$$
Now, we get a $\Lambda$-map $\left(\begin{smallmatrix} {\left[\begin{smallmatrix} \sigma \\ h_1 \end{smallmatrix}\right]} \\ {\left[\begin{smallmatrix} \alpha\pi \\ 0 \end{smallmatrix}\right]} \end{smallmatrix}\right): \left[\begin{smallmatrix} E_Y \\ Y \end{smallmatrix}\right]_{\psi}\to \left[\begin{smallmatrix} E_J \\ J \end{smallmatrix}\right]_{\eta}\oplus\left[\begin{smallmatrix} I \\ 0 \end{smallmatrix}\right]$ and a $\Lambda$-map $(\left[\begin{smallmatrix} \delta \\ h_2 \end{smallmatrix}\right], \left[\begin{smallmatrix} \beta \\ 0 \end{smallmatrix}\right]): \left[\begin{smallmatrix} E_J \\ J \end{smallmatrix}\right]_{\eta}\oplus\left[\begin{smallmatrix} I \\ 0 \end{smallmatrix}\right]\to\left[\begin{smallmatrix} E_{Y'} \\ Y' \end{smallmatrix}\right]_{\psi'}$ with composition  $$(\left[\begin{smallmatrix} \delta \\ h_2 \end{smallmatrix}\right], \left[\begin{smallmatrix} \beta \\ 0 \end{smallmatrix}\right])\left(\begin{smallmatrix} {\left[\begin{smallmatrix} \sigma \\ h_1 \end{smallmatrix}\right]} \\ {\left[\begin{smallmatrix} \alpha\pi \\ 0 \end{smallmatrix}\right]} \end{smallmatrix}\right)=\left[\begin{smallmatrix} \delta\sigma \\ h_2h_1 \end{smallmatrix}\right]+\left[\begin{smallmatrix} \beta\alpha\pi \\ 0 \end{smallmatrix}\right]=\left[\begin{smallmatrix} \delta\sigma+\beta\alpha\pi \\ h_2h_1 \end{smallmatrix}\right]=\left[\begin{smallmatrix} f \\ h \end{smallmatrix}\right].$$ This shows that $\left[\begin{smallmatrix} f \\ h \end{smallmatrix}\right]: \left[\begin{smallmatrix} E_Y \\ Y \end{smallmatrix}\right]_{\psi}\to \left[\begin{smallmatrix} E_{Y'} \\ Y' \end{smallmatrix}\right]_{\psi'}$ factors through the injective object
$\left[\begin{smallmatrix} E_J \\ J \end{smallmatrix}\right]_{\eta}\oplus\left[\begin{smallmatrix} I \\ 0 \end{smallmatrix}\right]$ of $\mathscr{S}(A, M, B)$, and hence $j_*(h) = \overline{\left[\begin{smallmatrix} f\\
h\end{smallmatrix}\right]} = 0$ in $\overline{\mathscr{S}(A, M, B)}$. Thus $j_*$ induces a functor $$\overline{j_*}: B\mbox{-}\overline{\rm mod}\to\overline{\mathscr{S}(A, M, B)}$$ given by $\overline{j_*}(Y): = \left[\begin{smallmatrix} E_Y \\ Y \end{smallmatrix}\right]_{\psi}$ and $\overline{j_*}(\overline{h}): = \overline{\left[\begin{smallmatrix} f \\ h \end{smallmatrix}\right]}$.

\vskip5pt

By construction it is clear that $\overline{j_*}$ is  full. For $h: Y\to Y'$, assume that $\overline{j_*}(\overline{h})  = \overline{\left[\begin{smallmatrix} f \\ h \end{smallmatrix}\right]}=0$.
Then  $\left[\begin{smallmatrix} f \\ h \end{smallmatrix}\right]: \left[\begin{smallmatrix} E_Y \\ Y \end{smallmatrix}\right]_{\psi}\to \left[\begin{smallmatrix} E_{Y'} \\ Y' \end{smallmatrix}\right]_{\psi'}$ factors through an injective object $\left[\begin{smallmatrix} E_J \\ J \end{smallmatrix}\right]_{\eta}\oplus\left[\begin{smallmatrix} I \\ 0 \end{smallmatrix}\right]$ of $\mathscr{S}(A, M, B)$, and hence $h: Y\to Y'$ factors through the injective $B$-module $J$. So $\overline{j_*}$ is also faithful.

\vskip5pt

It remains to prove that $(\overline{j^*}, \overline{j_*})$ is an adjoint pair.  Thus, for $\left[\begin{smallmatrix} X \\ Y \end{smallmatrix}\right]_{\phi}\in\mathscr{S}(A, M, B)$ and $Y'\in\modcat{B}$, we need to show that there is a bi-functorial isomorphism
$\overline{\Hom}_B(\overline{j^*}(\left[\begin{smallmatrix} X \\ Y \end{smallmatrix}\right]_{\phi}), Y')\cong\overline{\Hom}_{\Lambda}(\left[\begin{smallmatrix} X \\ Y \end{smallmatrix}\right]_{\phi}, \overline{j_*}(Y'))$, i.e. $$\overline{\Hom}_B(Y, Y')\cong\overline{\Hom}_{\Lambda}(\left[\begin{smallmatrix} X \\ Y \end{smallmatrix}\right]_{\phi}, \left[\begin{smallmatrix} E_{Y'} \\ Y' \end{smallmatrix}\right]_{\psi'})$$
where $\psi': M\otimes_BY'\to E_{Y'}$ is an injective envelope of $M\otimes_BY'$. We claim that the map $\overline h\mapsto \overline{\left[\begin{smallmatrix} f\\
h\end{smallmatrix}\right]}$ gives such an isomorphism, where $f: X\rightarrow E_{Y'}$ is an $A$-map such that $f\phi= \psi'(1\otimes h)$.
In fact,  a $\Lambda$-map $\left[\begin{smallmatrix} f \\ h \end{smallmatrix}\right]: \left[\begin{smallmatrix} X \\ Y \end{smallmatrix}\right]_{\phi}\to\left[\begin{smallmatrix} E_{Y'} \\ Y' \end{smallmatrix}\right]_{\psi'}$ factors through an injective object  $\left[\begin{smallmatrix} I \\ 0 \end{smallmatrix}\right]\oplus\left[\begin{smallmatrix} E_{J} \\ J \end{smallmatrix}\right]_{\psi}$ of $\mathscr{S}(A, M, B)$  if and only if $h: Y\to Y'$ factors through the injective $B$-module $J$. This shows that the given map above is well-defined and injective; and by the construction it is clearly surjective. This completes the proof.
\end{proof}

\subsection{\bf Proof of Theorem \ref{theorem-recollement-additive-cat}}
To see the diagram $(1.1)$ forms a recollement of additive categories, by Lemmas \ref{corollary-adjoint-pairs} and \ref{lemma-fully-faithfu-j} it remains to prove that $\overline{i_*}$ and $\overline{j_!}$ are fully faithful, and that $\Ima\overline{i_*}=\Ker\overline{j^*}$.

\vskip5pt

Since $_AM_B$ satisfies the condition ({\rm IP}), by Proposition \ref{enoughprojinjobj}$(2)$,  $\left[\begin{smallmatrix} I \\ 0 \end{smallmatrix}\right]$ and $\left[\begin{smallmatrix}M\otimes_BJ\\ J \end{smallmatrix}\right]_{\rm Id}$ are injective objects of $\mathscr{S}(A, M, B)$, where $I$ is an injective $A$-module and $J$ is an injective $B$-module. Recall that $\overline{i_*}: A\mbox{-}\overline{\rm mod}\to\overline{\mathscr{S}(A, M, B)}$ is given by $X\mapsto \left[\begin{smallmatrix} X \\ 0 \end{smallmatrix}\right]$ and $\overline{f}\mapsto \overline{\left[\begin{smallmatrix} f \\ 0 \end{smallmatrix}\right]}$, and $\overline{j_!}: B\mbox{-}\overline{\rm mod}\to\overline{\mathscr{S}(A, M, B)}$ is given by $Y\mapsto \left[\begin{smallmatrix} M\otimes_BY \\ Y \end{smallmatrix}\right]_{\rm Id}$ and $\overline{h}\mapsto \overline{\left[\begin{smallmatrix} 1\otimes h \\ h \end{smallmatrix}\right]}$. It is easy to see that  $\overline{i_*}$ and $\overline{j_!}$ are fully faithful.

\vskip5pt

It is clear that $\Ker(\overline{j^*})=\{\left[\begin{smallmatrix} X \\ Y \end{smallmatrix}\right]_{\phi}\in\mathscr{S}(A, M, B) \ | \ _BY \ \mbox{injective} \}$. For any $\left[\begin{smallmatrix} X \\ Y \end{smallmatrix}\right]_{\phi}\in\Ker(\overline{j^*})$, by the condition ({\rm IP}),  $M\otimes_BY$ is an injective $A$-module, and hence the exact sequence $0\to M\otimes_BY\overset{\phi}\to X\to \Cok(\phi)\to 0$ splits. So $\left[\begin{smallmatrix} X \\ Y \end{smallmatrix}\right]_{\phi}\cong \left[\begin{smallmatrix} M\otimes_BY \\ Y \end{smallmatrix}\right]_{\rm Id}\oplus\left[\begin{smallmatrix} \Cok(\phi) \\ 0 \end{smallmatrix}\right]$. Since $\left[\begin{smallmatrix} M\otimes_BY \\ Y \end{smallmatrix}\right]_{\rm Id}$ is an injective object of $\mathscr{S}(A, M, B)$, $\left[\begin{smallmatrix} X \\ Y \end{smallmatrix}\right]_{\phi}=\left[\begin{smallmatrix} \Cok(\phi) \\ 0 \end{smallmatrix}\right]$ in $\overline{\mathscr{S}(A, M, B)}$. Thus  $\Ker\overline{j^*}\subseteq\Ima\overline{i_*}$, and  $\Ima\overline{i_*}=\Ker\overline{j^*}$.

\vskip5pt

If both $A$ and $B$ are selfinjective, then by Lemma \ref{corollary-adjoint-pairs},
all the functors $\overline{i^*}, \overline{i_*}, \overline{i^!}, \overline{j_!}$ and $\overline{j^*}$ are triangle functors, and hence $\overline{j_*}$ is also a triangle functor. So $(1.1)$ is a recollement of triangulated categories. In this case $\Lambda$ is a Gorenstein algebra and $\mathscr{S}(A, M, B)$ is exactly the category of Gorenstein-projective $\Lambda$-modules ([Z2, Lemma 2.1, Thm. 2.2]), thus $\overline{\mathscr{S}(A, M, B)}$ is  exactly the singularity category $\mathcal D^b_{sg}(\Lambda): = \mathcal D^b(\Lambda\mbox{-}{\rm mod})/K^b({\rm proj}(\Lambda))$ (see R. Buchweitz [Buch, Thm. 4.4.1]; see also [O]). Similarly,
$A\mbox{-}\overline{\rm mod}\cong \mathcal D^b_{sg}(A)$ and $B\mbox{-}\overline{\rm mod}\cong \mathcal D^b_{sg}(B)$. $\square$

\section{\bf The dual version: the epimorphism category induced by a bimodule} \subsection{} We briefly state the dual version for the next section. {\it The epimorphism category $\mathscr{F}(A, M, B)$ induced by a bimodule} $_AM_B$ is the subcategory of $\modcat{\Lambda}$ consisting of $\left[\begin{smallmatrix}
X\\ Y
\end{smallmatrix}\right]_{\phi}$ such that $\eta_{_{Y, X}}(\phi): Y\to\Hom_A(M, X)$ is an epic left $B$-map, where
$$\eta_{_{Y, X}}: \Hom_A(M\otimes_BY, X)\cong \Hom_B(Y, \Hom_A(M, X))$$ is the adjunction isomorphism. Then $\mathscr{F}(A, M, B)$ contains all the injective $\Lambda$-modules and is closed under direct sums, direct summands and extensions. Thus $\mathscr{F}(A, M, B)$ is a Krull-Schmidt exact category with the canonical exact structure.

\vskip5pt

\noindent {\bf Proposition 2.3$^\prime$} {\it Let $_AM_B$ be a bimodule with $_AM$ projective. Then

$(1)$ \ $\mathscr{F}(A, M, B)$ has enough injective objects$,$
which are exactly injective $\Lambda$-modules.

$(2)$ \ $\mathscr{F}(A, M, B)$ has enough projective objects$;$ and the projective objects of $\mathscr{F}(A, M, B)$ are exactly  $\left[\begin{smallmatrix} P \\ C \end{smallmatrix}\right]_{\eta_{C, P}^{-1}(\theta)}$ and $\left[\begin{smallmatrix} 0 \\ Q \end{smallmatrix}\right]$, where $P$ $($resp. $Q$$)$ runs over
projective $A$-modules $($resp. $B$-modules$)$, and $\theta: C\to\Hom_A(M, P)$ is a projective cover of the left $B$-module $\Hom_A(M, P)$.
In particular, if in addition $\Hom_A(M, A)$ is a projective left $B$-module, then the projective objects of $\mathscr{F}(A, M, B)$ are exactly $\left[\begin{smallmatrix} P \\ \Hom_A(M, P) \end{smallmatrix}\right]_{\varphi}$ and $\left[\begin{smallmatrix} 0 \\ Q \end{smallmatrix}\right]$.}

\vskip5pt

\noindent {\bf Corollary 2.4$^\prime$} \ {\it Let $_AM_B$ be a bimodule with $_AM$ projective.
Then $\mathscr{F}(A,M,B)$ is a {\rm Frobenius} category $($with the canonical exact structure$)$ if and only if $A$ and $B$ are selfinjective algebras and $M_B$ is projective.}

\vskip5pt

\noindent {\bf Theorem 1.1$^\prime$} \ {\it Let $M$ be an $A$-$B$-bimodule. Then

\vskip5pt

$(1)$ \ The following are equivalent$:$

\hskip20pt ${\rm (i)}$ \ $_AM$ is projective$;$

\hskip20pt ${\rm (ii)}$ \  $\mathscr{F}(A, M, B)$ is a coresolving subcategory of $\modcat{\Lambda};$

\hskip20pt ${\rm (iii)}$ \ $L:=\left[\begin{smallmatrix} A \\ C \end{smallmatrix}\right]_{\eta_{C, A}^{-1}(\theta)}\oplus\left[\begin{smallmatrix} 0 \\ B \end{smallmatrix}\right]$ is the unique tilting left $\Lambda$-module, up to multiplicities of indecomposable direct summands, such that  $\mathscr{F}(A, M, B)=L^{\perp}$, where $C$ is a projective cover of the $B$-module $\Hom_A(M, A)$ with projection $\theta: C\to\Hom_A(M, A)$.

\vskip5pt

$(2)$ \ $\mathscr{F}(A, M, B)$ is a covariantly finite subcategory of $\Lambda$-{\rm mod}. Moreover, if $_AM$ is projective, then $\mathscr{F}(A, M, B)$ is a functorially finite subcategory of $\Lambda$-{\rm mod}, and has {\rm Auslander-Reiten} sequences}.

\vskip5pt

\noindent {\bf Corollary 1.2$^\prime$} \ {\it If $_AM$ is projective and $\Hom_A(M, A)$ is a projective left $B$-module, then $L=\left[\begin{smallmatrix}A\\ \Hom_A(M, A) \end{smallmatrix}\right]_\varphi\oplus\left[\begin{smallmatrix}0\\ B\end{smallmatrix}\right]_{\rm Id}$ is a cotilting left $\Lambda$-module such that $\mathscr{F}(A, M, B)=L^\perp $.}

\vskip5pt

\noindent {\bf Theorem \ref{theorem-recollement-additive-cat}$^\prime$.} \  {\it Let $_AM_B$ be a bimodule such that $_AM$ is a projective $A$-module and $\Hom_A(M, A)$ is a projective left $B$-module.
Then there is a recollement of additive categories

\begin{center}
\begin{picture}(130,35)
\put(-10,20){\makebox(-22,2) {$B\mbox{-}\underline{\rm mod}$}}
\put(30,28){\vector(-1,0){30}}
\put(0,20){\vector(1,0){30}}
\put(30,12){\vector(-1,0){30}}
\put(50,20){\makebox(25,0.8) {$\underline{\mathscr{F}(A, M, B)}$}}
\put(95,20){\vector(1,0){30}}
\put(125,12){\vector(-1,0){30}}
\put(125,28){\vector(-1,0){30}}
\put(135,20){\makebox(25,0.5){$A\mbox{-}\underline{\rm mod}.$}}
\put(15,35){\makebox(3,1){\scriptsize$$}}
\put(15,24){\makebox(3,1){\scriptsize$$}}
\put(15,14){\makebox(3,1){\scriptsize$$}}
\put(109,35){\makebox(3,1){\scriptsize$$}}
\put(109,24){\makebox(3,1){\scriptsize$$}}
\put(109,14){\makebox(3,1){\scriptsize$$}}
\end{picture}
\end{center}
\vspace{-10pt}\noindent If in addition $A$ and $B$ are selfinjective algebras, then it is in fact a recollement of singularity categories.}

\vskip5pt

Note that if $_AM_B$ is an exchangeable  bimodule then $\Hom_A(M, A)$ is a projective left $B$-module.

\subsection{Another description of the epimorphism category induced by a bimodule} Recall that the right module version of $\mathscr S(A, M, B)$ is $\mathscr S(A, M, B)_r$, which is the subcategory of {\rm mod}$\Lambda$ consisting of the triple $(U, V)_\psi$, where $U\in {\rm mod}A$,
$V\in {\rm mod}B$, and $\psi: U\otimes_A M \rightarrow V$ is a monic $B$-map.

\begin{prop} \label{smonsepi}  The restriction of $\D: {\rm mod}\Lambda \rightarrow \Lambda\mbox{-}{\rm mod}$ gives a duality
$\D: \mathscr S (A, M, B)_r \rightarrow \mathscr F(A, M, B).$
\end{prop}
\begin{proof} For a right $A$-module $U$, denote by $\alpha_{_{U}}$ the adjunction isomorphism
$$\D(U\otimes_AM) = \Hom_R(U\otimes_AM, J) \cong \Hom_A(M, \Hom_R(U, J)) = \Hom_A(M, \D U)$$
where $\D = \Hom_R(-, J)$ is the duality. For  a left $A$-module $X$ and a left $B$-module $Y$, denote by $\eta_{_{Y, X}}$ the adjunction isomorphism
$\Hom_A(M\otimes_BY, X) \cong \Hom_B(Y, \Hom_A(M, X)).$

\vskip5pt
For a right $\Lambda$-module $(U, V)_\psi$ with a right $B$-map $\psi: U\otimes_AM \rightarrow V$, we have $\D V \stackrel{\D(\psi)}\rightarrow \D(U\otimes_AM)\stackrel {\alpha_{_{U}}}  \rightarrow \Hom_A(M, \D U)$, and $\eta^{-1}_{_{\D V, \D U}}: \Hom_B(\D V, \Hom_A(M, \D U)) \rightarrow \Hom_A(M\otimes_B \D V, \D U).$
Then $$\D: {\rm mod}\Lambda \rightarrow \Lambda\mbox{-}{\rm mod}, \ (U, V)_\psi\mapsto \left[\begin{smallmatrix}\D U \\ \D V \end{smallmatrix}\right]_{\eta^{-1}_{_{\D V, \D U}}(\alpha_{_{U}}\D(\psi))}$$
with $\eta^{-1}_{_{\D V, \D U}}(\alpha_{_{U}}\D(\psi)): M\otimes_B \D V \rightarrow \D U$. For a left $\Lambda$-module $\left[\begin{smallmatrix} X \\ Y \end{smallmatrix}\right]_\phi$ with a left $A$-map $\phi: M\otimes_BY \rightarrow X$,
we have $Y \stackrel {\eta_{_{Y, X}}(\phi)}\rightarrow \Hom_A(M, X)\stackrel {\alpha^{-1}_{_{\D X}}} \rightarrow \D(\D X \otimes_AM).$
Then a quasi-inverse of $\D: {\rm mod}\Lambda \rightarrow \Lambda\mbox{-}{\rm mod}$ is
$$\D: \Lambda\mbox{-}{\rm mod} \rightarrow {\rm mod}\Lambda, \ \ \left[\begin{smallmatrix} X \\ Y \end{smallmatrix}\right]_\phi\mapsto (\D X, \D Y)_{\D(\alpha^{-1}_{_{\D X}}\eta_{_{Y, X}}(\phi))}$$
with $\D(\alpha^{-1}_{_{\D X}}\eta_{_{Y, X}}(\phi)): \D X \otimes_AM \rightarrow \D Y.$ \
In fact,
\begin{align*}\D\D(U, V)_\psi & = \D \left[\begin{smallmatrix}\D U \\ \D V \end{smallmatrix}\right]_{\eta^{-1}_{_{\D V, \D U}}(\alpha_{_{U}}\D(\psi))}
\cong (U, V)_{_{\D(\alpha^{-1}_{_U} \eta_{_{\D V, \D U}}(\eta^{-1}_{_{\D V, \D U}}(\alpha_{_U}\D(\psi))))}} = (U, V)_{\psi}\end{align*}
and \ $\D\D\left[\begin{smallmatrix} X \\ Y \end{smallmatrix}\right]_\phi  = \D (\D X, \D Y)_{\D(\alpha^{-1}_{_{\D X}}\eta_{_{Y, X}}(\phi))}
\cong \left[\begin{smallmatrix} X \\ Y \end{smallmatrix}\right]_{_{\eta^{-1}_{_{Y, X}}(\alpha_{_{\D X}} \D \D(\alpha^{-1}_{_{\D X}}\eta_{_{Y, X}}(\phi)))}} = \left[\begin{smallmatrix} X \\ Y \end{smallmatrix}\right]_\phi.$
\vskip5pt

If a right $\Lambda$-module $(U, V)_\psi\in \mathscr S (A, M, B)_r,$ i.e., $\psi: U\otimes_A M \rightarrow V$ is a monic right $B$-map, then
$$\D(U, V)_\psi = \left[\begin{smallmatrix}\D U \\ \D V \end{smallmatrix}\right]_{\eta^{-1}_{_{\D V, \D U}}(\alpha_{_{U}}\D(\psi))}\in \mathscr F(A, M, B),$$ since $\eta_{_{\D V, \D U}}(\eta^{-1}_{_{\D V, \D U}}(\alpha_{_{U}}\D(\psi))) = \alpha_{_{U}}\D(\psi): \D V \rightarrow \Hom_A(M, \D U)$ is
an epic left $B$-map.

\vskip5pt

If a left $\Lambda$-module $\left[\begin{smallmatrix} X \\ Y \end{smallmatrix}\right]_\phi\in \mathscr F(A, M, B),$ i.e.,
$\eta_{_{Y, X}}(\phi): Y\rightarrow\Hom_A(M, X)$ is an epic left $B$-map, then
$\D\left[\begin{smallmatrix} X \\ Y \end{smallmatrix}\right]_\phi = (\D X, \D Y)_{\D(\alpha^{-1}_{_{\D X}}\eta_{_{Y, X}}(\phi))}\in \mathscr S (A, M, B)_r,$ since $_{\D(\alpha^{-1}_{_{\D X}}\eta_{_{Y, X}}(\phi))}: \D X\otimes_A M \rightarrow \D Y$ is monic.
\end{proof}

\section{\bf Ringel-Schmidmeier-Simson equivalences via cotilting modules}

\begin{defn} \label{rss} \ A {\rm Ringel-Schmidmeier-Simson} equivalence, in short, an {\rm RSS} equivalence, induced by a bimodule $_AM_B$ is an equivalence $F\colon \mathscr S(A, M, B)\cong \mathscr F(A, M, B)$ of categories,  such that $$F\left[\begin{smallmatrix}
   X \\
   0
\end{smallmatrix}\right] \cong \left[\begin{smallmatrix} X \\ \Hom_A(M, X)\end{smallmatrix}\right]_{\varphi}, \ \forall \ X\in A\mbox{-}{\rm mod},  \ \ \ \mbox{and} \ \ \ F\left[\begin{smallmatrix}
   M\otimes_B Y  \\
   Y
\end{smallmatrix}\right]_{\rm Id} \cong \left[\begin{smallmatrix} 0 \\ Y \end{smallmatrix}\right], \ \forall \ Y\in B\mbox{-}{\rm mod},$$
where $\varphi: M\otimes_B\Hom_A(M, X)\longrightarrow X$ is the involution map.
 \end{defn}

An {\rm RSS} equivalence implies a strong symmetry. It was first observed by C. Ringel and M. Schimdmeier [RS2] for $_AM_B  \ = \ _AA_A$. For an {\rm RSS} equivalence induced by a chain without relations we refer to D. Simson [S1-S3], and for {\rm RSS} equivalences induced by acyclic quivers with monomial relations we refer to [ZX].

\subsection{Special cotilting modules induced by exchangeable  bimodules}
Let $M$ be an $A$-$B$-bimodule.
If $M\otimes_B\D B$ is an injective left $A$-module, then by Corollary \ref{cotilting2}, $_\Lambda T=\left[\begin{smallmatrix}\D(A_A)\\ 0 \end{smallmatrix}\right]\oplus\left[\begin{smallmatrix}M\otimes_B \D(B)\\ \D(B_B) \end{smallmatrix}\right]_{\rm Id}$ is a cotilting left $\Lambda$-module with $\mathscr{S}(A, M, B)={}^\perp T$. If $\D(A)\otimes_AM$ is an injective right $B$-module, then by a right module version of Theorem \ref{cotilting} (cf. Subsection 2.4),  $U_\Lambda =(\D(_AA), \D(_AA)\otimes_AM)_{\rm Id}\oplus (0, \D(_BB))$ is a cotilting right $\Lambda$-module with $\mathscr{S}(A, M, B)_r = {}^\perp (U_\Lambda)$. If $_AM_B$ is exchangeable, then both the conditions are satisfied; and  by $\D(_A A_A)\otimes_AM\cong M\otimes_B\D(_B B_B)$, we can regard that $T$ and $U$ have the same underlying abelian group $\D A\oplus (M\otimes_B\D B)\oplus \D B$. The following lemma claims that in this case $T$ (and $U$) can be endowed with a $\Lambda$-$\Lambda$-bimodule structure such that $\mathscr{S}(A, M, B)={}^\perp (_\Lambda T)$ and $\mathscr{S}(A, M, B)_r = {}^\perp (T_\Lambda)$ (thus $T_\Lambda\cong U_\Lambda$).

\begin{lem} \label{RSS0} \ Let $_AM_B$ be an exchangeable  bimodule with an $A$-$B$-bimodule isomorphism $g: \D(_A A_A)\otimes_AM\cong M\otimes_B\D(_B B_B)$, and $_\Lambda T=\left[\begin{smallmatrix}\D A\\ 0 \end{smallmatrix}\right]\oplus\left[\begin{smallmatrix}M\otimes_B \D B\\ \D B \end{smallmatrix}\right]_{\rm Id}$. Then $T$ has a $\Lambda$-$\Lambda$-bimodule structure such that both $_\Lambda T$ and $T_\Lambda$ are cotilting modules, $T_\Lambda \cong (\D(_AA), \D A\otimes_A M)_{\rm Id}\oplus (0, \D(_B B))$ as right $\Lambda$-modules,  $\mathscr{S}(A, M, B)={}^\perp (_\Lambda T)$ and $\mathscr S (A, M, B)_r  =  \ ^\perp (T_\m).$

\vskip5pt
Write
$T = \left[\begin{smallmatrix}\D(A_A) & M\otimes_B \D(B_B)\\ 0 & \D(B_B) \end{smallmatrix}\right]$. Then the right $\Lambda$-action is given by
$$\left(\begin{smallmatrix}\alpha & m'\otimes_B \beta\\ 0 & \beta' \end{smallmatrix}\right)\left(\begin{smallmatrix}a & m\\ 0 & b \end{smallmatrix}\right)
= \left(\begin{smallmatrix}\alpha a & g(\alpha\otimes_A m) + m'\otimes_B \beta b \\ 0 & \beta' b \end{smallmatrix}\right).$$
\end{lem}

\begin{proof} \ By Example \ref{mainexample}$(1)$, $_AM_B$ satisfies the condition ${\rm (IP)}$. It follows from Corollary \ref{cotilting2} that $_\Lambda T$ is a cotilting left $\Lambda$-module with $\mathscr{S}(A, M, B)={}^\perp (_\Lambda T)$.

\vskip5pt

Since $g$ is a right $B$-isomorphism,  we get a right $\Lambda$-isomorphism
$$({\rm Id}_{\D(_A A)}, g^{-1}): (\D(_A A), M\otimes_B \D B)_g\cong (\D(_A A), \D A\otimes_A M)_{\rm Id}.$$ We endow the abelian group
$T=\D A\oplus (M\otimes_B\D B)\oplus \D B$ with a right $\Lambda$-module structure via $g$, i.e., $T_\Lambda = (\D(_AA), M\otimes_B \D B)_g\oplus (0, \D(_B B))$. Then we get a right $\Lambda$-isomorphism:
$$T_\Lambda = (\D(_AA), M\otimes_B \D B)_g\oplus (0, \D(_B B))\cong (\D(_AA), \D A\otimes_A M)_{\rm Id}\oplus (0, \D(_B B)).$$
since $g$ is also a left $A$-map, one can easily verify that $T$ is a $\Lambda$-$\Lambda$-bimodule. We omit the details.

\vskip5pt

By Example \ref{mainexample}$(1)$,  $\D(_A A_A)\otimes_AM$ is an injective right $B$-module. It follows from the right module version of Theorem \ref{cotilting} (cf. Subsection 2.4) that $T_\Lambda\cong (\D(_AA), \D A\otimes_A M)_{\rm Id}\oplus (0, \D(_B B))$ is  a cotilting right $\Lambda$-module with $\mathscr S (A, M, B)_r  =  \ ^\perp (T_\m).$
\end{proof}
The following fact will play a crucial role in proving the existence of an {\rm RSS} equivalence.
\begin{lem} \label{RSS} \ Let $_AM_B$ be an exchangeable  bimodule, and $_\Lambda T_\Lambda$ the $\Lambda$-$\Lambda$-bimodule $T$ given in {\rm Lemma \ref{RSS0}}. Then
there is an algebra isomorphism $\rho: \End_\Lambda(_\Lambda T)^{op}\cong \Lambda$, such that under $\rho$, the right module
$T_{\End_\Lambda(_\Lambda T)^{op}}$ coincides with the right module $T_\Lambda;$ and there is an algebra isomorphism $\End_\Lambda(T_\Lambda)\cong \Lambda$,
such that under this algebra isomorphism, the left module $_{\End_\Lambda(T_\Lambda)}T$ coincides with the left module $_\Lambda T.$\end{lem}

\begin{proof} \  Since $T_\Lambda$ is a right $\Lambda$-module, we get the canonical  algebra homomorphism $\rho: \m \longrightarrow {\rm End}_\m(_\m T)^{op}$,
$\lambda \mapsto ``t\mapsto t\lambda"$.  By Lemma \ref{RSS0}, $T_\Lambda$ is cotilting, so
$T_\Lambda$ is faithful (Suppose $T\lambda = 0$ for $\lambda\in \Lambda$. By a surjective $\Lambda$-map $T_0 \longrightarrow \D \Lambda$ with $T_0\in {\rm add} T$, we see $(\D \Lambda)\lambda = 0$, i.e., $\D(\lambda \Lambda) = 0$. So $\lambda = 0$). Thus $\rho$ is an injective map. On the other hand, we have algebra isomorphisms
\begin{align*}{\rm End}_\m(_\m T)^{op} & \cong \left[\begin{smallmatrix}
{\rm End}_\m(\left[\begin{smallmatrix}\D A\\ 0 \end{smallmatrix}\right])&\Hom_\Lambda(\left[\begin{smallmatrix}M\otimes_B {\D B}\\ \D B \end{smallmatrix}\right]_{\rm Id}, \left[\begin{smallmatrix}\D A\\ 0 \end{smallmatrix}\right])\\ \Hom_\Lambda(\left[\begin{smallmatrix}\D A\\ 0 \end{smallmatrix}\right], \left[\begin{smallmatrix}M\otimes_B {\D B}\\ \D B\end{smallmatrix}\right]_{\rm Id})&{\rm End}_\m(\left[\begin{smallmatrix}M\otimes_B {\D B}\\ \D B \end{smallmatrix}\right]_{\rm Id})
\end{smallmatrix}\right]^{op} \\& \cong \left[\begin{smallmatrix} A^{op}&0\\ M &B^{op}\end{smallmatrix}\right]^{op}
\cong  \left[\begin{smallmatrix}
A&M\\0&B
\end{smallmatrix}\right] = \Lambda.\end{align*}
Denote this algebra isomorphism $\End_\Lambda(_\Lambda T)^{op}\cong \Lambda$ by $h$. Since $\Lambda$ is an Artin $R$-algebra, where $R$ is a commutative
artinian ring, $h\rho: \Lambda \longrightarrow \Lambda$ is an $R$-endomorphism of artinian $R$-module $\Lambda$. Since $h \rho$ is an injective map, it follows that
$h\rho$ is surjective, and hence $\rho$ is surjective (since $h$ is an $R$-module isomorphism).  Thus $\rho$ is an algebra isomorphism. By the construction of $\rho$,
$T_{{\rm End}_\m(_\m T)^{op}}$ is exactly $T_\m.$

\vskip5pt

Since $\rho: {\rm End}_\m(_\m T)^{op}\cong \Lambda$ as algebras, and under $\rho$,
$T_{{\rm End}_\m(_\m T)^{op}}$ is exactly $T_\m.$ By the tilting theory, the homomorphism $\Lambda \longrightarrow \End_\Lambda(T_\Lambda)$ given by
$\lambda\mapsto ``t\mapsto \lambda t"$  is an algebra isomorphism ([HR, p.409]), and hence  $_{{\rm End}_{\m}(T_{\m})} T$ is exactly $_\m T$ (one can also prove this by the same argument as above).
\end{proof}

\subsection {\bf Existence of RSS equivalences.} For any left $\Gamma$-module $L$,  following [AR], let $\mathcal X_{_ {_{\Gamma}} L}$ be the subcategory of $\Gamma$-mod
consisting of $\Gamma$-modules $_\Gamma X$ such that there is an exact sequence
$$0\longrightarrow X \longrightarrow L_0 \stackrel {f_0}\longrightarrow L_1\longrightarrow \cdots \longrightarrow L_j \stackrel {f_j} \longrightarrow L_{j+1} \longrightarrow \cdots $$
with $L_j\in {\rm add} (_\Gamma L)$ and ${\rm Im}f_j\in \ ^\perp (_\Gamma L)$ for $j\ge 0$.  The following fact is in T. Wakamatsu [W, Prop. 1].

\begin{lem} \label{wakamatsu} \ {\rm ([W])}  For any $\Gamma$-module $_\Gamma L$ with $C: = \End_\Gamma(L)^{op}$, we have a contravariant functor
$$\Hom_\Gamma(-, \ _\Gamma L): \mathcal X_{_ {_{\Gamma}}L}\longrightarrow \ ^\perp (L_C),$$
such that if $_\Gamma X\in \mathcal X_{_ {_{\Gamma}}L}$, then the canonical $\Gamma$-map $\ _\Gamma X \rightarrow {\rm Hom}_{C}(\Hom_\Gamma(_\Gamma X, \ _\Gamma L), \ L_C)$ is an isomorphism.
\end{lem}

\noindent {\bf Proof of Theorem \ref{RSS1}.} \ {\bf Step 1.} We first prove that $\D\Hom_\m(-, \ _\m T_\Lambda):
\mathscr S(A, M, B)\longrightarrow \mathscr F(A, M, B)$ is an equivalence of categories.
By Proposition \ref{smonsepi}, \ $\D:\mathscr S (A, M, B)_r \longrightarrow \mathscr F(A, M, B)$ is a duality.
So, it suffices to prove that $\Hom_\m(-, \ _\m T_\Lambda):
\mathscr S(A, M, B)\longrightarrow \mathscr S(A, M, B)_r$ is a duality.

\vskip5pt

By Lemma \ref{RSS0}, $T$ is a $\Lambda$-$\Lambda$-bimodule such that $\mathscr S(A, M, B) = \ ^\perp (_{\m} T)$ and $\mathscr S (A, M, B)_r  =  \ ^\perp (T_\m).$
Since $_\m T$ is cotilting, it follows from
M. Auslander and I. Reiten [AR, Thm. 5.4(b)] that $\mathcal X_{_ {_{\m}} T} = \ ^\perp (_{\m} T).$  Thus
$\mathscr S(A, M, B) = \ ^\perp (_{\m} T) = \mathcal X_{_ {_{\m}} T}.$

\vskip5pt

By Lemma \ref{RSS},  there is an algebra isomorphism $\rho: \End_\Lambda(_\Lambda T)^{op}\cong \Lambda$, such that
$T_{\End_\Lambda(_\Lambda T)^{op}} = T_\Lambda$  under $\rho$. So we can apply Lemma \ref{wakamatsu} to $_\Lambda T$ to get a contravariant functor: $$\Hom_\m(-, \ _\m T): \ \mathcal X_{_ {_{\m}} T} = \mathscr S(A, M, B) \longrightarrow \mathscr S (A, M, B)_r=  \ ^\perp (T_\m)$$
such that for $X\in \mathscr S(A, M, B) = \mathcal X_{_ {_{\m}} T}$, the canonical left $\m$-map $\ _{\m}X \longrightarrow {\rm Hom}_{\m}(\Hom_\m(_\m X, \ _\m T), \ T_{\m})$ is an isomorphism.

\vskip5pt

Similarly,  $T_\m $ is a cotilting module and $\mathscr S (A, M, B)_r = \ ^\perp (T_\m) = \mathcal X_{T_ {_{\m}}}$ ([AR, Thm. 5.4(b)]).
By Lemma \ref{RSS}, there is an algebra isomorphism $\End_\Lambda(T_\Lambda)\cong \Lambda$,
such that $_{\End_\Lambda(T_\Lambda)}T = \ _\Lambda T$ under this isomorphism. So we can apply the right module version of Lemma \ref{wakamatsu} to $T_\Lambda$
to get a contravariant functor
$$\Hom_\m(-, T_\m): \mathcal X_{T_ {_{\m}}}=\mathscr S (A, M, B)_r \longrightarrow \mathscr S (A, M, B) = \ ^\perp (_\m T)$$
such that for each $Y\in \mathscr S (A, M, B)_r = \mathcal X_{T_ {_{\m}}}$,
the canonical left $\m$-map $Y_{\m}\longrightarrow {\rm Hom}_\m(\Hom_{\m}(Y_\m, \ T_\m), \ _\m T)$ is an isomorphism. Thus
$\Hom_\m(-, \ _\m T): \ \mathscr S (A, M, B) \longrightarrow \mathscr S (A, M, B)_r$
is a duality with a quasi-inverse $\Hom_\m(-, T_\m): \mathscr S (A, M, B)_r \longrightarrow \mathscr S (A, M, B)$.

\vskip5pt

{\bf Step 2.} \ Put $F: = \D\Hom_\m(-, \ _\m T_\Lambda): \mathscr S (A, M, B) \cong \mathscr F (A, M, B)$  (cf. Proposition \ref{smonsepi}). It remains to prove
$F(\left[\begin{smallmatrix}
   X \\
   0
\end{smallmatrix}\right]) \cong \left[\begin{smallmatrix} X \\ \Hom_A(M, X)\end{smallmatrix}\right]_\varphi$ and $F(\left[\begin{smallmatrix}
   M\otimes_B Y  \\
   Y
\end{smallmatrix}\right]_{\rm Id}) \cong \left[\begin{smallmatrix} 0 \\ Y \end{smallmatrix}\right]$ for $X\in A\mbox{-}{\rm mod}$ and $Y\in B\mbox{-}{\rm mod}$.
This is true by direct computations, which are included as an Appendix at the end of this paper. $\square$

\subsection{RSS equivalences and the Nakayama functors}
It is natural to ask when the Nakayama functor $\mathcal N_A= \D\Hom_A(-, \ _AA_A): \ A\mbox{-}{\rm mod}\longrightarrow A\mbox{-}{\rm mod}$ induces
an {\rm RSS} equivalence $\mathscr S(A, M, B)\cong \mathscr F(A, M, B)$. Denote by
$\mathcal {GP}(\Lambda)$ (resp. $\mathcal {GI}(\Lambda)$) the subcategory of $\Lambda$-mod consisting of Gorenstein-projective (resp. Gorenstein-injective) modules ([EJ]).

\begin{prop} \label{NakayamaandRSS} \ Let $_AM_B$ be an exchangeable  $A$-$B$-bimodule. Then the restriction of the {\rm Nakayama} functor
$\mathcal N_\m$ gives an {\rm RSS} equivalence $\mathscr S(A, M, B)\cong \mathscr F(A, M, B)$ if and only if both $A$ and $B$ are {\rm Frobenius} algebras.
If this is the case, we have $\mathscr S(A, M, B)= \mathcal {GP}(\Lambda)$ and $\mathscr F(A, M, B)= \mathcal {GI}(\Lambda)$. \end{prop}
\begin{proof} \ If the restriction of $\mathcal N_\m$ gives an {\rm RSS} equivalence, then we have $\mathcal N_\m\left[\begin{smallmatrix}
   A \\
   0
\end{smallmatrix}\right] = \left[\begin{smallmatrix} _AA \\ \Hom_A(M, A)\end{smallmatrix}\right]_{\rm Id} \ \mbox{and} \  \mathcal N_\m\left[\begin{smallmatrix}
   M \\
   B
\end{smallmatrix}\right]_{\rm Id} = \left[\begin{smallmatrix} 0 \\ _BB \end{smallmatrix}\right].$
However
$$\mathcal N_\m\left[\begin{smallmatrix}
   A \\
   0
\end{smallmatrix}\right] = \D\Hom_\m(\left[\begin{smallmatrix}
   A \\
   0
\end{smallmatrix}\right], \Lambda) = \D\Hom_\m(\left[\begin{smallmatrix}
   A \\
   0
\end{smallmatrix}\right], \left[\begin{smallmatrix}A \\ 0\end{smallmatrix}\right]\oplus \left[\begin{smallmatrix} M\\ B \end{smallmatrix}\right]_{\rm Id}) = \D (A, M)_{\rm Id}= \left[\begin{smallmatrix} \D(A_A) \\ \D M\end{smallmatrix}\right]_{\rm Id},$$
and
$$\mathcal N_\m\left[\begin{smallmatrix}
   M  \\
   B
\end{smallmatrix}\right]_{\rm Id} = \D\Hom_\m(\left[\begin{smallmatrix}
   M \\
   B
\end{smallmatrix}\right]_{\rm Id}, \Lambda) = \D\Hom_\m(\left[\begin{smallmatrix}
   M \\
   B
\end{smallmatrix}\right]_{\rm Id}, \left[\begin{smallmatrix}A \\ 0\end{smallmatrix}\right]\oplus \left[\begin{smallmatrix} M\\ B \end{smallmatrix}\right]_{\rm Id}) =  \D(0, B) = \left[\begin{smallmatrix} 0 \\ \D(B_B) \end{smallmatrix}\right].$$ So $\D(A_A)\cong \ _AA$ and $\D(B_B) \cong \ _BB$, i.e., $A$ and $B$ are Frobenius algebras.

\vskip5pt

Conversely, if both $A$ and $B$ are Frobenius algebras, then replacing $_\Lambda T_\Lambda$ by $_\Lambda\Lambda_\Lambda$ and using the same arguments as in the proofs of Lemmas \ref{RSS0} and \ref{RSS} and Theorem \ref{RSS1}, we get an {\rm RSS} equivalence $$\mathcal N_\m = \D\Hom_\Lambda(-, \ _\Lambda\Lambda_\Lambda): \mathscr S(A, M, B)\longrightarrow \mathscr F(A, M, B).$$

If both $A$ and $B$ are Frobenius algebras and both $_AM$ and $M_B$ are projective, then $\Lambda$ is a Gorenstein algebra (cf. [Z2, Lemma 2.1]), and then by [Z2, Thm. 2.2] and its dual, $\mathscr S(A, M, B)= \mathcal {GP}(\Lambda)$ and $\mathscr F(A, M, B)= \mathcal {GI}(\Lambda)$. This completes the proof. \end{proof}

{\bf Remark.} For any algebra $\Gamma$, the Nakayama functor always gives an equivalence $\mathcal N_\Gamma: \mathcal {GP}(\Gamma)\cong \mathcal {GI}(\Gamma)$ (see [Bel, Prop. 3.4]).
Also, if $A$ and $B$ are Frobenius, then there is a left $A$-isomorphism $f_1: \D(A_A)\cong \ _AA$
and a left $B$-isomorphism $g_1: \D(B_B) \cong \ _BB$, so
$_\Lambda T =\left[\begin{smallmatrix}\D(A_A)\\ 0 \end{smallmatrix}\right]\oplus\left[\begin{smallmatrix}M\otimes_B \D(B_B)\\ \D(B_B) \end{smallmatrix}\right]_{\rm Id}\cong \left[\begin{smallmatrix}_AA\\ 0 \end{smallmatrix}\right]\oplus\left[\begin{smallmatrix}_AM\\ B \end{smallmatrix}\right]_{\rm Id}= \ _\Lambda\Lambda.$ By the symmetry a Frobenius algebra, there is a right $A$-isomorphism $f_2: \D(_AA)\cong \ A_A$
and a right $B$-isomorphism $g_2: \D(_BB) \cong \ B_B$. So $T_\Lambda \cong (\D(_AA), \D(_AA)\otimes_A M)_{\rm Id} \oplus (0, \D(_BB))\cong (A_A, M_B)_{\rm Id} \oplus (0, B_B) = \ \Lambda_\Lambda.$
But $T\ncong \Lambda$ as $\Lambda$-$\Lambda$-modules in general, since a Frobenius algebra is not necessarily a symmetric algebra. Thus, the ``if part" of Proposition \ref{NakayamaandRSS} is not a corollary of Theorem \ref{RSS1}.

\subsection{} We illustrate Theorem \ref{RSS1} and Proposition \ref{NakayamaandRSS}. The conjunction of paths of a quiver is from the right to the left.

\vskip5pt

\begin{exm} \label{exchangeable1} Let $B$ be the path algebra $k(b\longrightarrow a)$.  We write the indecomposable $B$-modules as
$\begin{smallmatrix}0\\ 1\end{smallmatrix} = S(a) = P(a), \begin{smallmatrix}1\\ 1\end{smallmatrix} = P(b) = I(a), \begin{smallmatrix}1\\ 0\end{smallmatrix}= S(b) = I(b).$ Let $A: = B\oplus B$ and $_AM_B: = \ _AA_B$. So $_AM_B$ is an exchangeable  bimodule, and
$\Lambda: = \left[\begin{smallmatrix}A & M \\ 0 & B \end{smallmatrix}\right] = \left[\begin{smallmatrix}B & 0 & B \\ 0 & B & B \\ 0 & 0 & B \end{smallmatrix}\right]= B\otimes_k kQ,$ where $Q$ is the
quiver $1 \stackrel \alpha\longleftarrow 3 \stackrel \beta
\longrightarrow 2$. Thus $\m$ is given by the quiver
\[\xymatrix @R=0.6cm@C=0.8cm{4\ar[d]_-{\gamma_1}  & 6\ar[d]^-{\gamma_3}\ar[l]_-{\alpha'}  \ar[r]^-{\beta'}& 5\ar[d]^-{\gamma_2}\\
1& 3\ar[l]_-{\alpha}\ar[r]^-{\beta}&
2}\] with relations \ $\gamma_1\alpha' - \alpha\gamma_3, \ \ \beta\gamma_3-\gamma_2\beta'$.
We will write a $\Lambda$-module as a representation of $Q$ over algebra $B$ $($see e.g. {\rm [ZX], [LZ])}.
Thus a $\Lambda$-module is written as $X_1 \stackrel {X_\alpha}\longleftarrow X_3 \stackrel {X_\beta}\longrightarrow X_2$, where $X_1, X_2, X_3\in B$-mod, and $X_\alpha$ and $X_\beta$ are $B$-maps.
For example, the indecomposable projective $\Lambda$-module $P(6)=\begin{smallmatrix}1&1&1\\1&1&1\end{smallmatrix}$  at vertex $6$ is
$$(\begin{smallmatrix}1\\ 1\end{smallmatrix}\stackrel {{\rm Id}}\longleftarrow \begin{smallmatrix}1\\ 1\end{smallmatrix}
\stackrel {{\rm Id}}\longrightarrow \begin{smallmatrix}1\\ 1\end{smallmatrix}) = (P(b)\stackrel {{\rm Id}}\longleftarrow P(b)
\stackrel {{\rm Id}}\longrightarrow P(b))$$
With this notation, the {\rm Auslander-Reiten} quiver of $\m$ is
$$
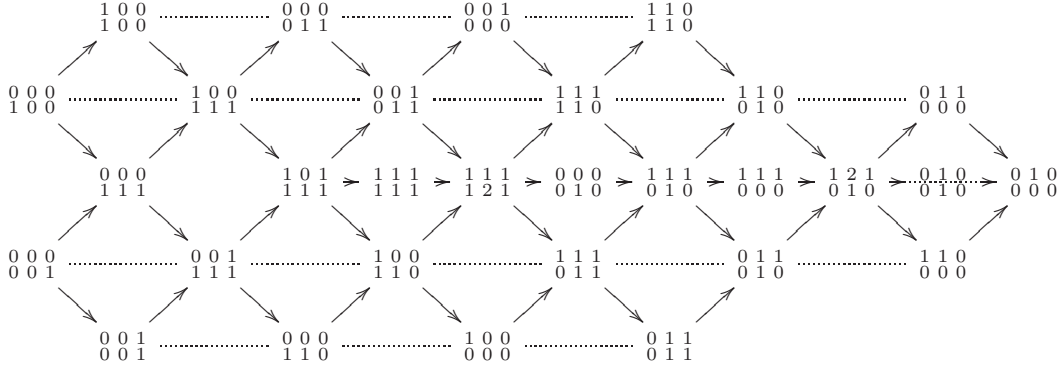
\xymatrix@R=0.45cm@C=0.2cm
{&
{\begin{smallmatrix}1&0&0\\1&0&0\end{smallmatrix}}\ar[dr] & &
{\begin{smallmatrix}0&0&0\\0&1&1\end{smallmatrix}}\ar[dr]\ar@{.}[ll] & &
{\begin{smallmatrix}0&0&1\\0&0&0\end{smallmatrix}}\ar[dr]\ar@{.}[ll]& &
{\begin{smallmatrix}1&1&0\\1&1&0\end{smallmatrix}}\ar[dr]\ar@{.}[ll] &&
\\
{\begin{smallmatrix}0&0&0\\1&0&0\end{smallmatrix}}\ar[ur]\ar[dr] & &
{\begin{smallmatrix}1&0&0\\1&1&1\end{smallmatrix}}\ar[dr]\ar[ur]\ar@{.}[ll]& &
{\begin{smallmatrix}0&0&1\\0&1&1\end{smallmatrix}}\ar[dr]\ar[ur]\ar@{.}[ll]& &
{\begin{smallmatrix}1&1&1\\1&1&0\end{smallmatrix}}\ar[dr]\ar[ur]\ar@{.}[ll]& &
{\begin{smallmatrix}1&1&0\\0&1&0\end{smallmatrix}}\ar[dr]\ar@{.}[ll]& &
{\begin{smallmatrix}0&1&1\\0&0&0\end{smallmatrix}}\ar[dr]\ar@{.}[ll] &
\\
& {\begin{smallmatrix}0&0&0\\1&1&1\end{smallmatrix}}\ar[dr]\ar[ur]& &
{\begin{smallmatrix}1&0&1\\1&1&1\end{smallmatrix}}\ar[r]\ar[ur]\ar[dr]
& {\begin{smallmatrix}1&1&1\\1&1&1\end{smallmatrix}}\ar[r]&
{\begin{smallmatrix}1&1&1\\1&2&1\end{smallmatrix}}\ar[r]\ar[dr]\ar[ur]
&{\begin{smallmatrix}0&0&0\\0&1&0\end{smallmatrix}}\ar[r] &
{\begin{smallmatrix}1&1&1\\0&1&0\end{smallmatrix}}\ar[r]\ar[ur]\ar[dr]
& {\begin{smallmatrix}1&1&1\\0&0&0\end{smallmatrix}}\ar[r]&
{\begin{smallmatrix}1&2&1\\0&1&0\end{smallmatrix}}\ar[r]\ar[ur]\ar[dr] &
{\begin{smallmatrix}0&1&0\\0&1&0\end{smallmatrix}}\ar[r]&
{\begin{smallmatrix}0&1&0\\0&0&0\end{smallmatrix}}\ar@{.}[ll]
\\
{\begin{smallmatrix}0&0&0\\0&0&1\end{smallmatrix}}\ar[dr]\ar[ur] & &
{\begin{smallmatrix}0&0&1\\1&1&1\end{smallmatrix}}\ar[ur]\ar[dr] \ar@{.}[ll]& &
{\begin{smallmatrix}1&0&0\\1&1&0\end{smallmatrix}}\ar[dr]\ar[ur]\ar@{.}[ll]&&
{\begin{smallmatrix}1&1&1\\0&1&1\end{smallmatrix}}\ar[ur]\ar[dr]\ar@{.}[ll]&  &
{\begin{smallmatrix}0&1&1\\0&1&0\end{smallmatrix}}\ar[ur]\ar@{.}[ll]&  &
{\begin{smallmatrix}1&1&0\\0&0&0\end{smallmatrix}}\ar[ur]\ar@{.}[ll]
\\& {\begin{smallmatrix}0&0&1\\0&0&1\end{smallmatrix}}\ar[ur] & &
{\begin{smallmatrix}0&0&0\\1&1&0\end{smallmatrix}}\ar[ur] \ar@{.}[ll]& &
{\begin{smallmatrix}1&0&0\\0&0&0\end{smallmatrix}}\ar[ur]\ar@{.}[ll] & &
{\begin{smallmatrix}0&1&1\\0&1&1\end{smallmatrix}}\ar[ur]\ar@{.}[ll]}$$

Since $\Lambda$ is of the form $\left[\begin{smallmatrix}A & M \\ 0 & B \end{smallmatrix}\right]$, a $\Lambda$-module $X_1 \stackrel {X_\alpha}\longleftarrow X_3 \stackrel {X_\beta}\longrightarrow X_2$ is also written as a triple $\left[\begin{smallmatrix}
X_1\oplus X_2\\ X_3
\end{smallmatrix}\right]_{\phi}$, where $X_1\oplus X_2\in A$-mod and $\phi: M\otimes_BX_3\rightarrow X_1\oplus X_2$ is exactly the  $A$-map $\left(\begin{smallmatrix}X_\alpha &0\\0&X_\beta\end{smallmatrix}\right): X_3\oplus X_3\rightarrow X_1\oplus X_2$. Thus it is in $\mathscr S(A, M, B)$ if and only if
$X_\alpha   \ \mbox{and}   \ X_\beta  \  \mbox{are monic}.$ So the {\rm Auslander-Reiten} quiver of $\mathscr S(A, M, B)$  is$:$
$$
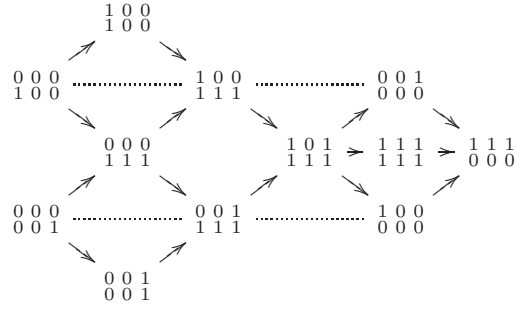
\xymatrix@R=0.25cm@C=0.2cm{ &
{\begin{smallmatrix}1&0&0\\1&0&0\end{smallmatrix}}\ar[dr]
 \\
{\begin{smallmatrix}0&0&0\\1&0&0\end{smallmatrix}}\ar[ur]\ar[dr] & &
{\begin{smallmatrix}1&0&0\\1&1&1\end{smallmatrix}}\ar[dr] \ar@{.}[ll]& &
{\begin{smallmatrix}0&0&1\\0&0&0\end{smallmatrix}}\ar[dr]\ar@{.}[ll]
\\
& {\begin{smallmatrix}0&0&0\\1&1&1\end{smallmatrix}}\ar[dr]\ar[ur]& &
{\begin{smallmatrix}1&0&1\\1&1&1\end{smallmatrix}}\ar[r]\ar[ur]\ar[dr] &
{\begin{smallmatrix}1&1&1\\1&1&1\end{smallmatrix}}\ar[r]&
{\begin{smallmatrix}1&1&1\\0&0&0\end{smallmatrix}}
\\
{\begin{smallmatrix}0&0&0\\0&0&1\end{smallmatrix}}\ar[ur]\ar[dr] & &
{\begin{smallmatrix}0&0&1\\1&1&1\end{smallmatrix}}\ar[ur]\ar@{.}[ll]& &
{\begin{smallmatrix}1&0&0\\0&0&0\end{smallmatrix}}\ar[ur]\ar@{.}[ll]
\\
& {\begin{smallmatrix}0&0&1\\0&0&1
\end{smallmatrix}}\ar[ur]}$$

Note that  $ (X_2\stackrel {X_\alpha}\twoheadrightarrow X_1
\stackrel {X_\beta}\twoheadleftarrow X_3)= \left[\begin{smallmatrix}
X_1\oplus X_2\\ X_3
\end{smallmatrix}\right]_{\phi}\in \mathscr F(A, M, B)$ if and only if $\binom{X_\alpha}{X_\beta}: X_3\longrightarrow X_1\oplus X_2$ is
an epic $B$-map $($in particular, $X_\alpha   \ \mbox{and}   \ X_\beta  \  \mbox{are epic};$ but this is not sufficient$)$. So the {\rm Auslander-Reiten} quiver of $\mathscr F(A, M, B)$ is:
$$
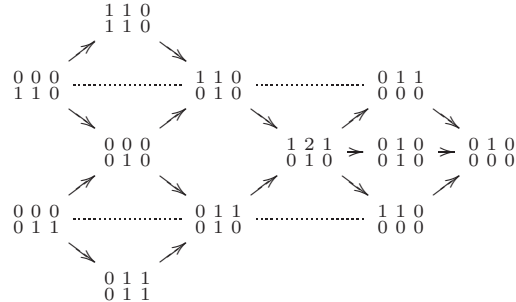
\xymatrix@R=0.25cm@C=0.2cm{ & {\begin{smallmatrix}1&1&0\\1&1&0\end{smallmatrix}}\ar[dr]
 \\
{\begin{smallmatrix}0&0&0\\1&1&0\end{smallmatrix}}\ar[ur]\ar[dr] & &
{\begin{smallmatrix}1&1&0\\0&1&0\end{smallmatrix}}\ar[dr] \ar@{.}[ll]& &
{\begin{smallmatrix}0&1&1\\0&0&0\end{smallmatrix}}\ar[dr]\ar@{.}[ll]
\\
& {\begin{smallmatrix}0&0&0\\0&1&0\end{smallmatrix}}\ar[dr]\ar[ur]& &
{\begin{smallmatrix}1&2&1\\0&1&0\end{smallmatrix}}\ar[r]\ar[ur]\ar[dr] &
{\begin{smallmatrix}0&1&0\\0&1&0\end{smallmatrix}}\ar[r]&
{\begin{smallmatrix}0&1&0\\0&0&0\end{smallmatrix}}
\\
{\begin{smallmatrix}0&0&0\\0&1&1\end{smallmatrix}}\ar[ur]\ar[dr] & &
{\begin{smallmatrix}0&1&1\\0&1&0\end{smallmatrix}}\ar[ur]\ar@{.}[ll]
& & {\begin{smallmatrix}1&1&0\\0&0&0\end{smallmatrix}}\ar[ur]\ar@{.}[ll]
\\
& {\begin{smallmatrix}0&1&1\\0&1&1
\end{smallmatrix}}\ar[ur]}$$

There is a unique {\rm RSS}
equivalence, sending an indecomposable object in $\mathscr S(A, M, B)$ to the one in $\mathscr F(A, M, B)$,
in the same positions of the {\rm Auslander-Reiten} quivers. Note that this {\rm RSS} equivalence is not given by the
{\rm Nakayama} functor $\mathcal N_\Lambda$, since it does not send projective $\Lambda$-modules to injective $\Lambda$-modules.
\end{exm}

\begin{exm} \label{exchangeable2} \ Let $\Lambda$ be the algebra given by the quiver
$\xymatrix{2\ar[r]^-{\beta}& 1\ar@(dr,ur)^-{\alpha}}$ with relation $\alpha^2$. Then $\Lambda: =\left[\begin{smallmatrix}A & M \\ 0 & B \end{smallmatrix}\right]$, where
$A:=k[\alpha]/\langle \alpha^2\rangle$, $B:=e_2\Lambda e_2\cong k$, and $_AM_k = e_1\Lambda e_2 = k\beta \oplus k\alpha\beta\cong \ _AA_k$. Then $_AM_k$ is an exchangeable  bimodule.
The {\rm Auslander-Reiten} quiver of $\Lambda$ is
$$
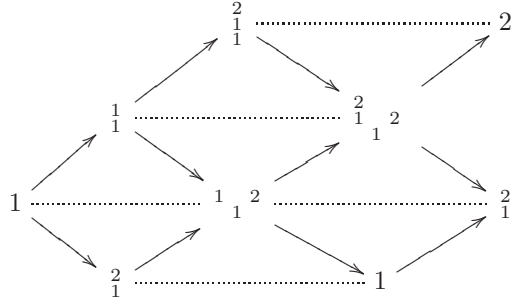
\xymatrix@R=0.4cm{& & {\begin{smallmatrix}2 \\ 1 \\ 1 \end{smallmatrix}}\ar[dr]& & 2\ar@{.}[ll]\\
& {\begin{smallmatrix} 1 \\ 1 \end{smallmatrix}}\ar[dr]\ar[ur] & & {\begin{smallmatrix} 2& &\\ 1 & &2 \\ & 1 &  & \end{smallmatrix}}\ar[dr]\ar[ur]\ar@{.}[ll] & \\
1\ar[dr]\ar[ur] & & {\begin{smallmatrix} 1  & & 2\\ & 1 & \end{smallmatrix}}\ar[dr]\ar[ur]\ar@{.}[ll] & & {\begin{smallmatrix}2 \\ 1 \end{smallmatrix}}\ar@{.}[ll] \\
& {\begin{smallmatrix}2 \\ 1 \end{smallmatrix}}\ar[ur] & & 1\ar[ur]\ar@{.}[ll] & }$$
where the two $1$'s represents the same module, and the two $\begin{smallmatrix}2 \\ 1 \end{smallmatrix}$'s also represents the same module.
To compute $\mathscr S(A, M, B)$ and $\mathscr F(A, M, B)$, we need to write a $\Lambda$-module $N$ as the form $\left[\begin{smallmatrix}
e_1N\\ e_2N
\end{smallmatrix}\right]_{\phi}$, where $\phi: M\otimes_k e_2N \longrightarrow e_1N$ is the $A$-map given by the $\Lambda$-actions:
$$\beta\otimes_k e_2n\mapsto \beta e_2n, \ \ \ \ \alpha\beta\otimes_ke_2n \mapsto \alpha\beta e_2n, \ \forall \ n\in N.$$
For example, the $\Lambda$-module $\begin{smallmatrix} 2 & &\\ 1 & &2 \\ & 1 &  & \end{smallmatrix}$ has a basis $e_1, \alpha, u, v$
with the $\Lambda$-actions

$$\begin{tabular}{|l|l|l|l|l|}
\hline
\multicolumn{1}{|c|}{}
&\multicolumn{1}{|c|}{$e_1$}
&\multicolumn{1}{|c|}{$\alpha$}
&\multicolumn{1}{|c|}{$u$}
&\multicolumn{1}{|c|}{$v$}
\\\hline
$e_1$         & $e_1$      & $\alpha$ & $0$         & $0$\\\hline
$e_2$         & $0$        & $0$      & $u$    & $v$\\\hline
$\alpha$      & $\alpha$   & $0$      & $0$         & $0$\\\hline
$\beta$       & $0$        & $0$      & $\alpha$    & $e_1$\\\hline
$\alpha\beta$ & $0$        & $0$      & $0$         & $\alpha$\\\hline
\end{tabular}$$

\vskip5pt
\noindent it follows that $\begin{smallmatrix} 2 & &\\ 1 & &2 \\ & 1 &  & \end{smallmatrix} = \left[\begin{smallmatrix}
ke_1\oplus k\alpha \\ ku\oplus kv \end{smallmatrix}\right]_{\phi}$ with the $A$-map $\phi$  given by the $\Lambda$-actions$:$
$$\beta\otimes_k u\mapsto \beta u = \alpha, \ \ \ \ \beta\otimes_k v\mapsto \beta v = e_1, \ \ \
\alpha\beta\otimes_k u\mapsto \alpha\beta u = 0, \ \ \ \alpha\beta\otimes_k v\mapsto \alpha\beta v = \alpha.$$
So $\phi$ is not monic and hence $\begin{smallmatrix} 2 & &\\ 1 & &2 \\ & 1 &  & \end{smallmatrix}\notin \mathscr S(A, M, B)$.
Also, since the $\Lambda$-module $\begin{smallmatrix} 1  & & 2\\ & 1 & \end{smallmatrix}$ has a basis $e_1, \alpha, u$ with the $\Lambda$-actions given by the table above, it follows that $\begin{smallmatrix} 1  & & 2\\ & 1 & \end{smallmatrix} = \left[\begin{smallmatrix}
ke_1\oplus k\alpha \\ ku \end{smallmatrix}\right]_{\phi}$ with the $A$-map $\phi$ given as above.
So $\phi$ is not monic and hence $\begin{smallmatrix} 1  & & 2\\ & 1 & \end{smallmatrix}\notin \mathscr S(A, M, B)$.
In this way we see that $\mathscr S(A, M, B)$ has $3$ indecomposable objects$:$ \ $1, \ \begin{smallmatrix} 1 \\ 1 \end{smallmatrix}, \ \begin{smallmatrix}2 \\ 1 \\ 1 \end{smallmatrix}$.

\vskip5pt

Similarly  a $\Lambda$-module $N=\left[\begin{smallmatrix}
e_1N\\ e_2N
\end{smallmatrix}\right]_{\phi}\in\mathscr F(A, M, B)$ if and only if $\varphi: e_2N \longrightarrow \Hom_A(M, e_1N) = e_1N$  is epic, where $\varphi$
is the image of $\phi: M\otimes_k e_2N \longrightarrow e_1N$ under the adjunction isomorphism.
Note that $\varphi: e_2N \longrightarrow e_1N$ is exactly given by the actions of $\beta$. For example,
since the $\Lambda$-module $\begin{smallmatrix} 2\\ 1 \end{smallmatrix}$ has a basis $\alpha, u$ with the $\Lambda$-actions given by the table above,
it follows that $\begin{smallmatrix} 2\\ 1 \end{smallmatrix}= \left[\begin{smallmatrix}
k\alpha \\ ku \end{smallmatrix}\right]_{\phi}$ with $\varphi: ku \longrightarrow k\alpha$ given by $u\mapsto\beta u = \alpha$.
So $\begin{smallmatrix} 2\\ 1 \end{smallmatrix}\in\mathscr F(A, M, B)$. In this way we see that $\mathscr F(A, M, B)$ has $3$ indecomposable objects$:$ $\begin{smallmatrix} 2 \\ 1 \end{smallmatrix}, \ \begin{smallmatrix} & 2 &\\ 1 & &2 \\ & 1 &  & \end{smallmatrix}, 2$.

\vskip5pt

There is a unique {\rm RSS}
equivalence given by $1\mapsto \begin{smallmatrix} 2 \\ 1 \end{smallmatrix}, \ \ \ \begin{smallmatrix} 1 \\ 1 \end{smallmatrix}
\mapsto \begin{smallmatrix} & 2 &\\ 1 & &2 \\ & 1 &  & \end{smallmatrix}, \ \ \ \begin{smallmatrix}2 \\ 1 \\ 1 \end{smallmatrix}\mapsto 2.$
Note that $\mathscr S(A, M, B)= \mathcal {GP}(\Lambda)$ and $\mathscr F(A, M, B)= \mathcal {GI}(\Lambda)$, and this {\rm RSS}
equivalence is given by the {\rm Nakayama} functor $\mathcal N_\Lambda$. \end{exm}

\subsection{}The following examples show that if $_AM_B$ is {\bf not} exchangeable, then the existence of an {\rm RSS} equivalence can not be guaranteed.
They also show that $\mathscr S(A, M, B)$ is {\bf not} the separated monomorphism category of
the corresponding quiver in the sense of [ZX], in general.

\begin{exm}\label{nonexchangeable} $(1)$ \ Let $A$ be the path algebra $k(2\longrightarrow 1)$, $B: = k$, and $_AM_k: = \ _A(Ae_1)_k = 0 1.$ Then $\Lambda: = \left[\begin{smallmatrix}A & M \\ 0 & k \end{smallmatrix}\right]= \left[\begin{smallmatrix}k & k & k \\ 0& k & 0 \\ 0 & 0 & k \end{smallmatrix}\right]$ is just the path algebra $kQ$, where $Q$ is the quiver $2\longrightarrow 1\longleftarrow 3$. The {\rm Auslander-Reiten} quiver of $\Lambda\mbox{-}{\rm mod}$ is
$$
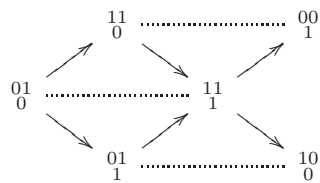
\xymatrix@R=0.3cm@C=0.6cm{& {\begin{smallmatrix}1 1\\ 0 \end{smallmatrix}}\ar[dr] && {\begin{smallmatrix}00\\ 1 \end{smallmatrix}}\ar@{.}[ll]
\\
{\begin{smallmatrix}01\\ 0\end{smallmatrix}}\ar[dr]\ar[ur] && {\begin{smallmatrix}11\\ 1\end{smallmatrix}}\ar[ur]\ar[dr]\ar@{.}[ll]
\\
& {\begin{smallmatrix}0 1\\ 1\end{smallmatrix}}\ar[ur]& & {\begin{smallmatrix}10\\ 0\end{smallmatrix}}\ar@{.}[ll]}$$
where we denote a $\Lambda$-module $V_2\longrightarrow V_1\longleftarrow V_3$ by $\begin{smallmatrix}{\rm dim}_kV_2 \ {\rm dim}_kV_1\\ {\rm dim}_kV_3\end{smallmatrix}.$
Since $$\D(_AA_A)\otimes_A M = \D(_AA_A)\otimes_A Ae_1\cong \D(e_1A) \ncong Ae_1 = M\otimes_B\D(_BB_B),$$  $_AM_k$ is {\bf not} an exchangeable  bimodule. In fact,
$_AM_k$ also does {\bf not} satisfy the condition {\rm (IP)}. The monomorphism category $\mathscr S(A, M, k)$ induced by $_AM_k$ is
$$\mathscr S(A, M, k) = \{\left[\begin{smallmatrix}X\\ Y\end{smallmatrix}\right]_\phi\in \Lambda\mbox{-}{\rm mod} \ | \ X\in A\mbox{-}{\rm mod}, \ Y\in k\mbox{-}{\rm mod}, \
M\otimes_kY\stackrel \phi \hookrightarrow X \ \mbox{is a monic} \ A\mbox{-map}\}.$$
Thus $\mathscr S(A, M, k)$ has $5$ indecomposable objects
$\begin{smallmatrix}01\\ 0\end{smallmatrix}, \ \begin{smallmatrix}11\\ 0\end{smallmatrix}, \ \begin{smallmatrix}01\\ 1\end{smallmatrix}, \ \begin{smallmatrix}1 1\\ 1\end{smallmatrix}, \ \begin{smallmatrix}1 0\\ 0\end{smallmatrix}.$
While the epimorphism category $\mathscr F(A, M, k)$ induced by $_AM_k$ is
\begin{align*}\mathscr F(A, M, k)& = \{\left[\begin{smallmatrix}X\\ Y\end{smallmatrix}\right]_\phi\in \Lambda\mbox{-}{\rm mod} \ | \ X\in A\mbox{-}{\rm mod}, \ Y\in k\mbox{-}{\rm mod}, \
Y\stackrel \varphi\twoheadrightarrow \Hom_A(M, X) \ \mbox{is an epic} \ k\mbox{-map}\}\\ & = \{\left[\begin{smallmatrix}X\\ Y\end{smallmatrix}\right]_\phi\in \Lambda\mbox{-}{\rm mod} \ | \ X\in A\mbox{-}{\rm mod}, \ Y\in k\mbox{-}{\rm mod}, \
Y\stackrel \varphi\twoheadrightarrow e_1X \ \mbox{is an epic} \ k\mbox{-map}\}\end{align*}
where $\varphi: = \eta_{X, Y}(\phi)$ and  $\eta_{X, Y}: \Hom_A(M\otimes_kY, X)\cong \Hom_k(Y, \Hom_A(M, X))$ is the adjunction isomorphism.
So $\mathscr F(A, M, k)$ has only $4$ indecomposable objects
$\begin{smallmatrix}01\\ 1\end{smallmatrix}, \ \begin{smallmatrix}11\\ 1\end{smallmatrix},  \ \begin{smallmatrix}0 0\\ 1\end{smallmatrix}, \ \begin{smallmatrix}1 0\\ 0\end{smallmatrix}.$
Thus $\mathscr S(A, M, k)\ncong\mathscr F(A, M, k).$

\vskip5pt

Note that the indecomposable objects of the separated monomorphism category ${\rm smon}(Q, 0, k)$ are exactly
the indecomposable projective $\Lambda$-modules. See {\rm [ZX, Exam. 2.3]}. So $\mathscr S(A, M, k)\ncong\mathscr {\rm smon}(Q, 0, k).$

\vskip5pt

$(2)$ \ Let $A$ and $B: = k$ be as in $(1)$, and $_AM_k: = S(2) = 1 0.$ Since $_AM$ is not projective, $_AM_k$ is {\bf not} exchangeable  $($note that $\D(_AA_A)\otimes_A M = \D(_AA_A)\otimes_A S(2)= \D(_AA_A)\otimes_A e_2S(2)
\cong \D(e_2A) = S(2) \cong M\otimes_B\D(_BB_B);$  and
$_AM_k$ satisfies the condition {\rm (IP)}$)$. Then $\Lambda: = \left[\begin{smallmatrix}A & M \\ 0 & k \end{smallmatrix}\right]= \left[\begin{smallmatrix}k & k & 0 \\ 0& k & k \\ 0 & 0 & k \end{smallmatrix}\right] \cong kQ/I$ with $Q = 3\stackrel \beta\longrightarrow 2\stackrel \alpha\longrightarrow 1$ and $I = \langle \alpha\beta\rangle$. The {\rm Auslander-Reiten} quiver of $\Lambda\mbox{-}{\rm mod}$ is
$$
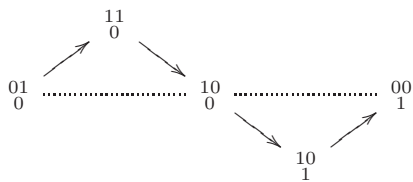
\xymatrix@R=0.3cm@C=0.6cm{& {\begin{smallmatrix}1 1\\ 0 \end{smallmatrix}}\ar[dr]
\\{\begin{smallmatrix}01\\ 0\end{smallmatrix}}\ar[ur] & & {\begin{smallmatrix}10\\ 0\end{smallmatrix}}\ar[dr]\ar@{.}[ll] && {\begin{smallmatrix}00\\ 1 \end{smallmatrix}}\ar@{.}[ll]
\\
& & & {\begin{smallmatrix}10\\ 1\end{smallmatrix}}\ar[ur]}$$
where a $\Lambda$-module $V_3\longrightarrow V_2\longrightarrow V_1$ is denoted by $\begin{smallmatrix}{\rm dim}_kV_2 \ {\rm dim}_kV_1\\ {\rm dim}_kV_3\end{smallmatrix}.$
The monomorphism category $\mathscr S(A, M, k) = \{\left[\begin{smallmatrix}X\\ Y\end{smallmatrix}\right]_\phi\in \Lambda\mbox{-}{\rm mod} \ | \ M\otimes_kY\stackrel \phi \hookrightarrow X \ \mbox{is a monic} \ A\mbox{-map}\}$
has $4$-indecomposable objects
$\begin{smallmatrix}01\\ 0\end{smallmatrix}, \ \begin{smallmatrix}11\\ 0\end{smallmatrix}, \ \begin{smallmatrix}10\\ 0\end{smallmatrix}, \ \begin{smallmatrix}1 0\\ 1\end{smallmatrix}.$
The epimorphism category $\mathscr F(A, M, k) = \{\left[\begin{smallmatrix}X\\ Y\end{smallmatrix}\right]_\phi\in \Lambda\mbox{-}{\rm mod} \ | \ Y\stackrel \varphi\twoheadrightarrow \Hom_A(S(2), X) \ \mbox{is an epic} \ k\mbox{-map}\}$
has also $4$ indecomposable objects
$\begin{smallmatrix}01\\ 0\end{smallmatrix}, \ \begin{smallmatrix}11\\ 0\end{smallmatrix},  \ \begin{smallmatrix}1 0\\ 1\end{smallmatrix}, \ \begin{smallmatrix}0 0\\ 1\end{smallmatrix}.$
We claim that there are no {\rm RSS} equivalences $F:\mathscr S(A, M, k)\cong\mathscr F(A, M, k).$ Otherwise, by the definition of an {\rm RSS} equivalence we get a contradiction
$$0\ne \Hom_\Lambda(\begin{smallmatrix}11\\ 0\end{smallmatrix}, \begin{smallmatrix}1 0\\ 1\end{smallmatrix})
\cong \Hom_\Lambda(F(\begin{smallmatrix}11\\ 0\end{smallmatrix}), F(\begin{smallmatrix}1 0\\ 1\end{smallmatrix}))
= \Hom_\Lambda(\begin{smallmatrix}11\\ 0\end{smallmatrix}, \begin{smallmatrix}0 0\\ 1\end{smallmatrix}) = 0.$$

\vskip5pt

By {\rm [ZX, Exam. 2.3]}, the indecomposable objects of the separated monomorphism category ${\rm smon}(Q, I, k)$ are exactly
the indecomposable projective $\Lambda$-modules.  So $\mathscr S(A, M, k)\ncong\mathscr {\rm smon}(Q, 0, k).$
\end{exm}

We propose the following problems.

\vskip5pt

{\bf 1.} \ What is a sufficient and necessary condition for the existence of an RSS equivalence?

\vskip5pt

{\bf 2.} \ Whether or not an RSS equivalence is unique?

\vskip10pt

\centerline {\bf Appendix: computations of Step 2 in the proof of Theorem \ref{RSS1}}

\vskip5pt

Put $F = \D\Hom_\m(-, T): \mathscr S (A, M, B) \cong \mathscr F (A, M, B)$. Let $X\in A\mbox{-}{\rm mod}$ and $Y\in B\mbox{-}{\rm mod}$. To prove
$F(\left[\begin{smallmatrix}
   X \\
   0
\end{smallmatrix}\right]) \cong \left[\begin{smallmatrix} X \\ \Hom_A(M, X)\end{smallmatrix}\right]_\varphi$ and
$F(\left[\begin{smallmatrix}
   M\otimes_B Y  \\
   Y
\end{smallmatrix}\right]_{\rm Id}) \cong \left[\begin{smallmatrix} 0 \\ Y \end{smallmatrix}\right]$, where $\varphi: M\otimes_B \Hom_A(M, X) \longrightarrow X$ is the involution map, by the proof of Proposition \ref{smonsepi}, it suffices to prove $$\Hom_\m(\left[\begin{smallmatrix}
   X \\
   0
\end{smallmatrix}\right], T) \cong (\D X, \D \Hom_A(M, X))_{\D(\alpha^{-1}_{\D X})} \eqno(6.1)$$
where $\alpha_{_{\D X}}: \D(\D X\otimes M) \cong \Hom_A(M, X)$ is the adjunction isomorphism, $\D(\alpha^{-1}_{\D X}): \D X\otimes_A M \longrightarrow \D\Hom_A(M, X)$ is a right $B$-map, and $$\Hom_\m(\left[\begin{smallmatrix}
   M\otimes_B Y  \\
   Y
\end{smallmatrix}\right]_{\rm Id}, T) \cong (0, \D Y).\eqno(6.2)$$

\vskip5pt

Put $e_1: = \left(\begin{smallmatrix}1 & 0\\ 0 & 0 \end{smallmatrix}\right),
\ e_2: = \left(\begin{smallmatrix}0 & 0\\ 0 & 1 \end{smallmatrix}\right)$. Then as left $\Lambda$-modules we have  $Te_1 = \left[\begin{smallmatrix}\D(A_A) \\ 0 \end{smallmatrix}\right]$
and $Te_2 = \left[\begin{smallmatrix}M\otimes_B \D B \\ \D B\end{smallmatrix}\right]_{\rm Id}\cong \left[\begin{smallmatrix} \D(A)\otimes_A M\\ \D(B_B) \end{smallmatrix}\right]_{g^{-1}}.$
For any left $\Lambda$-module $L$, $\Hom_\Lambda(L, T)$ is a right $\Lambda$-module. Then $\Hom_\Lambda(L, T)e_1\cong \Hom_\Lambda(L, Te_1)$ is a right $A$-module
and $\Hom_\Lambda(L, T)e_2\cong \Hom_\Lambda(L, Te_2)$ is a right $B$-module, in the obvious way, and  we have a right $\Lambda$-isomorphism
$$\Hom_\Lambda(L, T)\cong (\Hom_\Lambda(L, Te_1), \ \Hom_\Lambda(L, Te_2))_\phi \eqno(6.3)$$
where $\phi$ is explicit given by:
$$\phi: \Hom_\Lambda(L, Te_1)\otimes_A M \longrightarrow \Hom_\Lambda(L, Te_2), \ \ f\otimes_A m \mapsto ``l\mapsto f(l)
\left(\begin{smallmatrix}0 & m\\ 0 & 0 \end{smallmatrix}\right)e_2". \eqno (6.4)$$

\vskip5pt

Applying $(6.3)$ to $L = \left[\begin{smallmatrix}
   M\otimes_B Y  \\
   Y
\end{smallmatrix}\right]_{\rm Id}$ we get right $\Lambda$-isomorphisms:
\begin{align*}\Hom_\m(\left[\begin{smallmatrix}
    M\otimes_BY \\
   Y
\end{smallmatrix}\right]_{\rm Id}, T) &= (\Hom_\m(\left[\begin{smallmatrix}
    M\otimes_BY \\
   Y
\end{smallmatrix}\right]_{\rm Id}, Te_1),  \Hom_\m(\left[\begin{smallmatrix}
    M\otimes_BY \\
   Y
\end{smallmatrix}\right]_{\rm Id}, Te_2))\\& = (\Hom_\m(\left[\begin{smallmatrix}
    M\otimes_BY \\
   Y
\end{smallmatrix}\right]_{\rm Id}, \left[\begin{smallmatrix}\D(A_A) \\ 0\end{smallmatrix}\right]), \ \Hom_\m(\left[\begin{smallmatrix}
    M\otimes_BY \\
   Y
\end{smallmatrix}\right]_{\rm Id}, \left[\begin{smallmatrix} M\otimes_B \D(B_B)\\ \D(B_B) \end{smallmatrix}\right]_{\rm Id})\\ &
\cong (0, \ \Hom_B(Y, \D(B_B))\cong (0, \ \D Y).\end{align*}
This proves $(6.2)$.

\vskip5pt

Applying $(6.3)$ to $L = \left[\begin{smallmatrix}
   X   \\ 0 \end{smallmatrix}\right]$ we get right $\Lambda$-isomorphisms:
\begin{align*} \Hom_\m(\left[\begin{smallmatrix} X \\ 0 \end{smallmatrix}\right], T)
& = (\Hom_\m(\left[\begin{smallmatrix} X \\ 0 \end{smallmatrix}\right], Te_1), \Hom_\m(\left[\begin{smallmatrix} X \\ 0 \end{smallmatrix}\right], Te_2))
\\ & = (\Hom_\m(\left[\begin{smallmatrix} X \\ 0 \end{smallmatrix}\right], \left[\begin{smallmatrix}\D(A_A) \\ 0\end{smallmatrix}\right]), \Hom_\m(\left[\begin{smallmatrix} X \\ 0 \end{smallmatrix}\right], \left[\begin{smallmatrix} M\otimes_B \D(B_B)\\ \D(B_B) \end{smallmatrix}\right]_{\rm Id}))
\\& \cong (\D X, \ \Hom_\m(\left[\begin{smallmatrix}
   X \\
   0
\end{smallmatrix}\right], \left[\begin{smallmatrix} \D(A)\otimes M\\ \D(B_B) \end{smallmatrix}\right]_{g^{-1}}))\\ &
\cong  (\D X, \Hom_A(X, \D(A)\otimes_AM))_\phi\end{align*}
and by $(6.4)$, $\phi$ is explicitly given by
$$\phi: \D X\otimes_AM \longrightarrow \Hom_A(X, \D(A)\otimes_AM)), \ \ \alpha\otimes_A m\mapsto ``x\mapsto \alpha_x\otimes_A m" \eqno(6.5)$$
where $\alpha_x\in \D A$ sends $a$ to $\alpha(ax).$

\vskip5pt

To prove $(6.1)$, it is clear that we need a right $B$-isomorphism $\psi: \D\Hom_A(M, X)\cong \Hom_A(X, \D(A)\otimes_AM)$. For this, we need to use the assumption that $_AM$ is projective.
Without loss of the generality, one can take $_AM=Ae$ for some idempotent element $e\in A$.  First, we have group isomorphisms
\begin{align*}\D\Hom_A(M, X)  = \D\Hom_A(Ae, X) \cong \D(X)e \cong \Hom_A(X, \D(eA))\cong\Hom_A(X, \D(A)\otimes_AM).\end{align*}
This isomorphism $\psi: \D\Hom_A(M, X)\cong \Hom_A(X, \D(A)\otimes_AM)$ of abelian groups is explicitly given by
$$\gamma\mapsto ``x\mapsto \gamma_x\otimes_A e" \eqno (6.6)$$ where $\gamma_x\in \D A$ sends $a\in A$ to $\gamma(f_{a,x})$, and $f_{a,x}\in \Hom_A(M, X)$ sends $m = ce\in M = Ae$ to $max$.
We claim that $\psi$ is a right $B$-map, and hence $\psi$ is a right $B$-module isomorphism.

\vskip5pt
In fact, for each $b\in B$, suppose $eb = v_be\in M = Ae$ for some $v_b\in A$.  Then for each $x\in X$ we have
$$\psi(\gamma b)(x) = (\gamma b)_x\otimes e, \ \mbox{with} \ (\gamma b)_x\in \D A$$
and  $$(\psi(\gamma) b)(x) = \psi(\gamma)(x) b = \gamma_x\otimes_A eb = \gamma_x\otimes_A v_be = \gamma_x v_be\otimes_A e, \ \mbox{with} \ \gamma_x v_b\in \D A.$$
Thus, it suffices to show  $(\gamma b)_x(a) = (\gamma_x v_be)(a) = \gamma_x (v_bea)$ for each $a\in A$, i.e.,
$(\gamma b)(f_{a, x}) = \gamma (f_{v_bea, x})$. That is $\gamma (bf_{a, x}) = \gamma (f_{v_bea, x})$. This is really true, since both $bf_{a, x}$ and $f_{v_bea, x}$ sends $m$ to
$$(bf_{a, x})(m) = f_{a, x}(mb) = mbax=ce(eb)ax = ce(v_be)ax = mv_beax = f_{v_bea, x}(m).$$ This proves that $\psi: \D\Hom_A(M, X)\cong \Hom_A(X, \D(A)\otimes_AM)$ is a right $B$-module isomorphism.

\vskip5pt

So, we get the right $\Lambda$-isomorphism  $$({\rm Id}_{\D X}, \psi^{-1}): (\D X, \Hom_A(X, \D(A)\otimes_AM))_{\psi\D(\alpha^{-1}_{\D X})}\cong (\D X, \ \D\Hom_A(M, X))_{\D(\alpha^{-1}_{\D X})},$$
where $\alpha_{_{\D X}}: \D(\D X\otimes M) \cong \Hom_A(M, X)$ sends $\beta\in \D(\D X\otimes M)$ to $f\in \Hom_A(M, X)$
such that $$\beta(\alpha\otimes_A m) = \alpha f(m), \ \forall \ \alpha\in \D X, \ m\in M,$$
$\alpha^{-1}_{\D X}: \Hom_A(M, X) \longrightarrow \D(\D X\otimes_A M)$ is given by $f\mapsto ``\beta: \alpha\otimes_A m \mapsto \alpha f(m)",$ and $\D(\alpha^{-1}_{\D X}): \D X\otimes_A M \longrightarrow \D\Hom_A(M, X)$ is given by
$$\alpha\otimes_A m \mapsto ``\delta: f\mapsto \alpha f(m)",$$ $\psi: \D\Hom_A(M, X)\cong \Hom_A(X, \D(A)\otimes_AM)$ is  given by $(6.6)$,
and $\psi\D(\alpha^{-1}_{\D X}): \D X\otimes_A M \longrightarrow \Hom_A(X, \D(A)\otimes_AM).$ To prove $(6.1)$, it suffices to prove
$\phi = \psi\D(\alpha^{-1}_{\D X})$.
Thus by $(6.6)$ we have
$$(\psi\D(\alpha^{-1}_{\D X})(\alpha\otimes_A m))(x) = \psi(\delta)(x) = \delta_x\otimes_A e$$
where $\delta_x\in \D A$ sends $a\in A$ to $\delta(f_{a,x})= \alpha f_{a,x}(m) = \alpha(max)$. Comparing with $(6.5)$ we see $\delta_x = \alpha_x ce = \alpha_x m,$ since
$\alpha_x m(a) = \alpha_x (ma) = \alpha (max) = \delta_x(a).$
It follows that $$(\psi\D(\alpha^{-1}_{\D X})(\alpha\otimes_A m))(x) = \delta_x\otimes_A e = \alpha_x ce\otimes_A e=\alpha_x \otimes_A m = \phi(\alpha\otimes_A m)(x).$$ This proves $\psi\D(\alpha^{-1}_{\D X}) = \phi$, and hence completes the proof.

\bigskip

{\footnotesize

Bao-Lin Xiong \ \ \ \ {\tt Email: xiongbaolin@gmail.com}

Department of Mathematics, \ Beijing  University of Chemical  Technology,
Beijing 100029, P. R. China


\medskip

Pu Zhang \ \ \ \ {\tt Email: pzhang@sjtu.edu.cn}

School of Mathematics, \ Shanghai Jiao Tong University, Shanghai 200240, P. R. China

\medskip

Yue-Hui Zhang \ \ \ \ {\tt Email: zyh@sjtu.edu.cn}

School of Mathematics, \ Shanghai Jiao Tong University, Shanghai 200240, P. R. China}


\begin{thebibliography}{99}
\bibitem[ASS]{ASS} I. Assem, D. Simson,  A. Skowro\'nski,
Elements of the representation theory of associative algebras,
Vol.1. Techniques of representation theory, Lond. Math. Soc.
Students Texts 65, Cambridge University Press, 2006.
\bibitem[AR]{AR} M. Auslander, I. Reiten, Applications of
contravariantly finite subcategories, Adv. Math. 86(1991), 111-152.
\bibitem[ARS]{ARS} M. Auslander, I. Reiten, S. O.
Smal${\o}$, Representation Theory of Artin Algebras, Cambridge
Studies in Adv. Math. 36., Cambridge Univ. Press, 1995.
\bibitem[AS]{AS} M. Auslander, S. O. Smal{\o}, Almost
split sequences in subcategories, J. Algebra 69(1981), 426-454.
\bibitem[BBD]{BBD}A. A. Beilinson, J. Bernstein and P. Deligne, Faisceaux Perves, Ast¡äerique 100 (Soc. Math., France,
1982).
\bibitem[Bel]{Bel} A. Beligiannis, Cohen-Macaulay modules,
(co)torsion pairs and virtually Gorenstein algebras, J. Algebra
288(1)(2005), 137-211.
\bibitem[Bir]{Bir} G. Birkhoff, Subgroups of abelian groups, Proc. Lond. Math. Soc. II, Ser. 38(1934), 385-401.
\bibitem[Buch]{Buch} R.-O. Buchweitz, Maximal
Cohen-Macaulay modules and Tate cohomology over Gorenstein rings,
Unpublished manuscript, Hamburg (1987), 155pp.
\bibitem[C]{C} X. W. Chen, The stable monomorphism category of a Frobenius category, Math. Res. Lett. 18(1)(2011), 125-
137.
\bibitem[CZ]{CZ} X. W. Chen, P. Zhang, Quotient triangulated categories,
Manuscripta Mathematica, 123 (2007), 167-183.
\bibitem[EJ]{EJ} E. E.
Enochs, O. M. G. Jenda, Relative homological algebra, De Gruyter
Exp. Math. 30. Walter De Gruyter Co., 2000.
\bibitem[FZ]{FZ}  J. Feng,
P. Zhang, Types of  Serre subcategories of Grothendieck categories, J. Algebra (to appear)
\bibitem[FP]{FP}  V. Franjou,
T. Pirashvili, Comparison of abelian categories recollement, Doc.
Math. 9(2004), 41-56.
\bibitem[H1]{H1} D. Happel, Triangulated categories in
representation theory of finite dimensional algebras, London Math.
Soc. Lecture Notes Ser. 119, Cambridge Uni. Press, 1988.
\bibitem[H2]{H2} D. Happel, Reduction techniques for homological conjectures, Tsukuba J. Math. 17(1)(1993), 115-130.
\bibitem[HR]{HR} D. Happle, C. M. Ringel,  Tilted algebras, Trans. Amer. Math. Soc. 274(2)(1982), 399-443.
\bibitem[IKM]{IKM} O. Iyama, K. Kato, J.-I. Miyachi, Recollement on homotopy categories and Cohen¨CMacaulay modules, J. K-Theory 8(3)(2011), 507-541.
\bibitem[Ke] {Ke} B. Keller,  Derived categories and universal problems,
Comm. Algebra 19(1991), 699-747.
\bibitem[K\"o]{Ko}  S. K\"onig, Tilting complexes, perpendicular categories and recollements of derived module categories of rings, J. Pure Appl. Algebra
73(1991), 211-232.
\bibitem[KS]{KS} H. Krause, {\O}. Solberg, Applications of cotorsion pairs, J. London Math. Soc.(2)68(3)(2003), 631-650.
\bibitem[Ku]{Ku} N. J. Kuhn, Generic representation theory of the finite general linear groups and the steenrod algebra: II,
K-Theory 8(1994), 395-428.
\bibitem[KLM1]{KLM1} D. Kussin, H. Lenzing, H. Meltzer, Nilpotent operators and weighted projective lines, J. Reine Angew.
Math. 685(6)(2010), 33-71.
\bibitem[KLM2]{KLM2} D. Kussin, H. Lenzing, H. Meltzer, Triangle singularities, ADE-chains, and weighted projective lines, Adv.
Math. 237(2013), 194-251.
\bibitem[LZ]{LZ}   X. H. Luo,  P. Zhang, Separated monic
representations I: Gorenstein-projective modules, J. Algebra
479(2017), 1-34.
\bibitem[MV]{MV} R. MacPherson,  K. Vilonen, Elementary construction of perverse sheaves, Invent. Math. 84(1986), 403-436.
\bibitem[N]{N} A. Neeman, Triangulated
categories, Ann. Math. Studies 148,  Princeton University Press,
Princeton, NJ, 2001.
\bibitem[O] {O} \ \ D. Orlov, Triangulated categories of singularities and D-branes in Landau-Ginzburg models,
Proc. Steklov Inst. Math. 246(3)(2004), 227-248.
\bibitem[PS]{PS} B. J. Parshall, L. L. Scott, Derived categories, quasi-hereditary algebras and algebric groups, Math. Lecture Notes Series 3, Carleton Univ. (1988), 1-105.
\bibitem[PV]{PV} C. Psaroudakis, J. Vitoria, Recollement of module categories, Appl. Categ. Structures 22(2014), 579-593.
\bibitem [R]{Rin} C. M. Ringel, Tame algebras and integral
quadratic forms, Lecture Notes in Math. 1099, Springer-Verlag, 1984.
\bibitem[RS1]{RS1} C. M. Ringel, M. Schmidmeier, Submodules
categories of wild representation type, J. Pure Appl. Algebra
205(2)(2006), 412-422.
\bibitem[RS2]{RS2} C. M. Ringel, M. Schmidmeier, The Auslander-Reiten translation in submodule categories, Trans. Amer. Math. Soc. 360(2)(2008), 691-716.
\bibitem[RS3]{RS3} C. M. Ringel, M. Schmidmeier, Invariant subspaces of nilpotent operators I, J. rein angew. Math. 614(2008), 1-52.
\bibitem[S1]{S2} D. Simson, Representation types of the category of subprojective representations of a finite poset over $K[t]/(t^m)$
and a solution of a Birkhoff type problem, J. Algebra 311(2007), 1-30.
\bibitem[S2]{S3} D. Simson, Tame-wild dichotomy of Birkhoff type problems for nilpotent linear operators, J. Algebra 424(2015),
254-293.
\bibitem[S3]{S4} D. Simson, Representation-finite Birkhoff type problems for
nilpotent linear operators, to appear in: JPAA
\bibitem [W]{W} T. Wakamatsu,  On modules with trivial self-extensions, J. Algebra 114(1988),
106-114.
\bibitem[XZ]{XZ} B. L. Xiong, P. Zhang, Gorenstein-projective modules over triangular matrix Artin algebras,
J. Algebra Appl. 11(4)(2012), 1250066.
\bibitem[Z1]{Z1} P. Zhang, Monomorphism categories, cotilting theory,  and
Gorenstein-projective modules,  J. Algebra 339(2011), 181-202.
\bibitem[Z2]{Z2} P. Zhang, Gorenstein-projective modules and symmetric recollements, J. Algebra 388(2013), 65-80.
\bibitem[ZX]{ZX} P. Zhang, B. L. Xiong, Separated monic representations II: Frobenius subcategories and RSS equivalences, arXiv:1707.04866v1 [math.RT]
\end{thebibliography}
\end{document}